\documentclass{article}

\usepackage{amsmath, amsfonts, amsthm, amssymb,cite,color,graphicx,subcaption,listings}
\usepackage[margin = 1in]{geometry}

\newtheorem{thm}{Theorem}

\newtheorem{ex}[thm]{Example}
\newtheorem{lemma}[thm]{Lemma}
\newtheorem{prop}[thm]{Proposition}

\newtheorem{defn}[thm]{Definition}

\numberwithin{thm}{section}

\theoremstyle{remark}

\newtheorem{remark}{Remark}

\newcommand{\ignore}[1]{}

\newcommand{\N}{\mathbb{N}}
\newcommand{\R}{\mathbb{R}}

\newcommand{\B}{\mathbb{B}}

\newcommand{\Int}{\text{Int} \ }

\renewcommand{\Re}{\text{Re}}

\title{On Conditions for Rate-induced Tipping in Multi-Dimensional Dynamical Systems}

\author{Claire Kiers\footnote{Department of Mathematics, University of North Carolina at Chapel Hill} \ and Christopher K.R.T. Jones\footnotemark[\value{footnote}]}

\date{\normalsize\today}

\begin{document}

\maketitle

\abstract{The possibility of {\em rate-induced tipping} (R-tipping) away from an attracting fixed point has been thoroughly explored in 1-dimensional systems. In these systems, it is impossible to have R-tipping away from a path of quasi-stable equilibria that is {\em forward basin stable} (FBS), but R-tipping is guaranteed for paths that are non-FBS of a certain type. We will investigate whether these results carry over to multi-dimensional systems. In particular, we will show that the same conditions guaranteeing R-tipping in 1-dimension also guarantee R-tipping in higher dimensions; however, it is possible to have R-tipping away from a path that is FBS even in 2-dimensional systems. We will propose a different condition, {\em forward inflowing stability} (FIS), which we show is sufficient to prevent R-tipping in all dimensions. The condition, while natural, is difficult to verify in concrete examples.  {\em Monotone systems} are a class for which FIS is implied by an easily verifiable condition. As a result, we see how the additional structure of these systems makes predicting the possibility of R-tipping straightforward in a fashion similar to 1-dimension. In particular, we will prove that the FBS and FIS conditions in monotone systems reduce to comparing the relative positions of equilibria over time. An example of a monotone system is given that demonstrates how these ideas are applied to determine exactly when R-tipping is possible.}

\section{Introduction}\label{intro}

Tipping can be described as a sudden, drastic, irreversible change in the behavior of a solution as a result of a small change to the system. In part, tipping is interesting because it can be observed in nature. A recent example in the literature concerns the rise of temperature in peatlands (see \cite{compost}). There are different reasons that tipping can happen in a system; in particular \cite{pete} describes three types of tipping: bifurcation-, noise-, and rate-induced. This paper will focus on the third kind of tipping, which results from a fast parameter change in a dynamical system. It is the {\em rate} at which the parameter changes that causes the tipping, not simply the amount that it changes. For a thorough introduction into rate-induced tipping, the reader is encouraged to look in \cite{sebas}, but we will give a summary here that is sufficient for the rest of the paper. \par

Suppose we have an autonomous dynamical system
\begin{align}\label{autonomous}
\dot x & = f(x,\lambda)
\end{align}
where $x \in U$ for $U \subset \R^n$ open, $\lambda \in \R^m$, and $f \in C^2(U \times \R^m, \R^n)$. If we want to explore the possibility of rate-induced tipping in this system, we must allow the parameter $\lambda$ to change over time. Without loss of generality, we may assume that $\lambda \in \R$ because if not, we can parametrize each component of $\lambda$ with a different one-dimensional parameter. We want the parameter change to be bounded and sufficiently differentiable, so we choose $\Lambda \in C^2(\R,(\lambda_-,\lambda_+))$ for some $\lambda_- < \lambda_+$ satisfying
\begin{equation}
\begin{split}\label{star}
\lim_{s \to \pm \infty} \Lambda(s) &= \lambda_\pm \\
\lim_{s \to \pm \infty} \frac{d\Lambda}{ds} & = 0
\end{split}
\end{equation}
and obtain a corresponding non-autonomous system
\begin{align}\label{non-autonomous}
\dot x & = f(x,\Lambda(rt))
\end{align}
for some $r > 0$. The value $r$ can be thought of as the {\em rate} at which $\Lambda$ changes. When $r$ is small, the parameter change is gradual, and when $r$ is large, the parameter change is very sudden. We are interested in comparing the behavior of system \eqref{non-autonomous} for different values of $r$. \par

Since we prefer to work with autonomous systems, we introduce the variable $s = rt$ and augment system \eqref{non-autonomous} as
\begin{equation}
\begin{split}\label{augmented}
\dot x & = f(x,\Lambda(s)) \\
\dot s & = r.
\end{split}
\end{equation}
Notice that if $r = 0$, then \eqref{augmented} reduces to \eqref{autonomous} where $\lambda = \Lambda(s)$. \par

Suppose that for all $s \in \R$, $X(s)$ is an attracting equilibrium for the corresponding autonomous system \eqref{autonomous} with $\lambda = \Lambda(s)$. Then we say $(X(s),\Lambda(s))$ is a {\em stable path}. Define
$$X_\pm = \lim_{s \to \pm \infty} X(s).$$

As shown in \cite{sebas}, there is a unique trajectory $x^r(t)$ of \eqref{augmented} such that $x^r(t) \to X_-$ as $t \to -\infty$, which is the {\em local pullback attractor} to $X_-$. If $\lim_{t \to \infty} x^r(t) = X_+$, then we say that $x^r(t)$ {\em endpoint tracks} the stable path $(X(s),\Lambda(s))$. (Often we will just say {\em tracks} for short.) Using Fenichel's Theorem it can be shown that for all sufficiently small $r > 0$, $x^r(t)$ will endpoint track $(X(s),\Lambda(s))$. However, if $x^r(t) \not \to X_+$ as $t \to \infty$, then $x^r(t)$ does not endpoint track $(X(s),\Lambda(s))$, and we say that {\em rate-induced tipping} (or {\em R-tipping}) has occurred. This kind of tipping is also sometimes called {\em irreversible} rate-induced tipping. \par

Our interest is in showing what kinds of parameter changes $\Lambda$ can lead to the possibility of R-tipping for some $r > 0$. Some results are already known for 1-dimensional systems ($n=1$), and we will give these in Section \ref{1DSection}. These results are phrased in the language of {\em forward basin stability} or {\em forward basin stable} paths (FBS), so we will focus on ways that FBS (or lack thereof) relates to R-tipping in multi-dimensional systems ($n>1$). \par

In Section \ref{generalized}, we will give a constructive proof showing that R-tipping will happen in certain cases of no FBS, namely, if the position of a stable path $(X(s),\Lambda(s))$ at time $t_1$ is contained in the basin of attraction of a different stable path $(Y(s),\Lambda(s))$ at a later time $t_2$. We will look at the Lorenz `63 system and show how varying a parameter in a way that satisfies this condition leads to R-tipping there. In Section \ref{forwardBasinStable} we will give an example of a 2-dimensional system in which a path is FBS but the pullback attractor does not track it. In particular, this will show that FBS is not sufficient to prevent R-tipping in multi-dimensional systems. We will define a different condition, {\em forward inflowing stability} or {\em forward inflowing stable} paths (FIS), which is sufficient to prevent R-tipping away from a stable path. \par

In Section \ref{monotoneSection} we will focus on R-tipping in monotone systems. We will be able to use the results from Sections \ref{generalized} and \ref{forwardBasinStable} to give conditions for guaranteeing or preventing R-tipping that rely only on the relative positions of the equilibria in the system. For this reason, monotone systems are ideal systems for thinking about R-tipping. In Section \ref{exampleSection}, we will show how the methods described in this paper give a nearly complete characterization of the possibilities of R-tipping in a particular 2-dimensional monotone system. Finally in Section \ref{conclusion} we will have some discussion about how the method of FIS could apply to a broader range of examples than those explicitly covered here.

\section{R-Tipping in 1-Dimensional Systems}\label{1DSection}

We begin by giving the definition of forward basin stability and stating a result from \cite{sebas} about R-tipping in 1-dimensional systems (when $n = 1$) that we will reference in later sections. Unless explicitly stated, we will continue to use the notation from Section \ref{intro}.

\begin{defn}
For $s \in \R$, let $\B(X(s),\Lambda(s))$ be the basin of attraction of the stable equilibrium $X(s)$ for the autonomous system \eqref{autonomous} with $\lambda = \Lambda(s)$. A stable path $(X(s),\Lambda(s))$ is {\em forward basin stable} (FBS) if
$$\overline{\{X(u) : u < s\}} \subset \B(X(s),\Lambda(s)) \text{ for all } s \in \R.$$
\end{defn}

Then Theorem 3.2 of \cite{sebas} states:

\begin{thm}\label{1D}
Suppose we have a system of the form \eqref{augmented} for $n = 1$. Let $(X(s),\Lambda(s))$ be a stable path. Set $X_\pm = \lim_{s \to \pm \infty} X(s)$.
\begin{enumerate}
\item If $(X(s),\Lambda(s))$ is FBS, there can be no R-tipping away from $X_-$.
\item If there is another stable path $(Y(s),\Lambda(s))$ with $Y_\pm = \lim_{s \to \pm \infty} Y(s)$ such that $Y_+ \ne X_+$ and there are $u < v$ such that
$$X(u) \in \B(Y(v),\Lambda(v)),$$
then $(X(s),\Lambda(s))$ is not FBS and there is a parameter shift $\tilde \Lambda$ (which is a re-scaling of $\Lambda$) such that there is R-tipping away from $X_-$ for this $\tilde \Lambda$.
\item If there is a $Y_+ \ne X_+$ such that $Y_+$ is an attracting equilibrium of \eqref{autonomous} for $\lambda = \lambda_+$ and
$$X_- \in \B(Y_+,\lambda_+),$$
then $(X(s),\Lambda(s))$ is not FBS and there is R-tipping away from $X_-$ for this $\Lambda$.
\end{enumerate}
\end{thm}

In Sections \ref{generalized} and \ref{forwardBasinStable} we will see how the three parts of Theorem \ref{1D} do or do not generalize to systems with $n > 1$.

\section{Conditions to Guarantee R-Tipping}\label{generalized}

In this section we will prove that statements 2 and 3 of Theorem \ref{1D} generalize to multi-dimensional systems. First, we must establish some lemmas that will be useful later. In what follows, we will assume $(X(s),\Lambda(s))$ is a stable path of \eqref{augmented} with $X_- = \lim_{s \to -\infty} X(s)$. \par

This first lemma deals with the initial behavior of a pullback attractor to $X_-$:

\begin{lemma}\label{beginning}
Let $x^r(t)$ be the pullback attractor to $X_-$ in \eqref{augmented}. Given $\epsilon > 0$, there exists an $S > 0$ such that $x^r(t) \in B_\epsilon(X_-)$ when $rt < -S$.
\end{lemma}

The proof follows from making minor modifications to the proof of Theorem \ref{1D} in \cite{sebas}. In particular, since we want an $s$-value that tells us how far $x^r(t)$ has moved from $X_-$, we can think of the function $\delta(T)$ in \cite{sebas} instead as a function of $S$, $\delta(S)$. \par

Next, we discuss the end behavior of a trajectory of \eqref{augmented}. The purpose of Lemmas \ref{subsequence} - \ref{compactBasinSubset} is to show that if $X_+$ is an attracting fixed point of \eqref{autonomous} with $\lambda = \lambda_+$ and $\B(X_+,\lambda_+)$ is its basin of attraction, then any trajectory $x(t)$ of \eqref{augmented} that is in a compact subset of $\B(X_+,\lambda_+)$ for large enough $t$ will converge to $X_+$. In what follows, it will be helpful to distinguish between the flow of the augmented system \eqref{augmented} and the flow of the reduced systems \eqref{autonomous} for different values of $\lambda$. So, we will use the notation
$$x \cdot_{\lambda'} t$$
to denote a trajectory of \eqref{autonomous} with $\lambda = \lambda'$, while $x(t)$ will denote a trajectory in \eqref{non-autonomous} or \eqref{augmented}. \par

In autonomous systems we have the following property of omega limit sets (where $\omega(x) = \{y: x \cdot t_n \to y \text{ for some } t_n \to \infty\}$): If $z \in \omega(y)$ and $y \in \omega(x)$, then $z \in \omega(x)$. This next lemma states that, in a certain sense, this property holds in non-autonomous systems like \eqref{non-autonomous}. The proof is a simple application of the triangle inequality.

\begin{lemma}\label{subsequence}
Suppose $y \cdot_{\lambda_+} s_n \to z$ for some $\{s_n\} \to \infty$. If $x(t)$ is a trajectory of \eqref{non-autonomous} such that $x(t_n) \to y$ for some $\{t_n\} \to \infty$, then there exist $\{u_n\} \to \infty$ for which $x(u_n) \to z$.
\end{lemma}

If $p$ is an attracting fixed point of an autonomous system, there are arbitrarily small forward invariant neighborhoods of $p$. (This is sometimes shown in the proof of the Stable Manifold Theorem.) This next lemma shows that a similar statement is true for $X_+$ in \eqref{augmented}, where the forward invariant neighborhoods around $X_+$ extend both in the $x$- and $s$-dimensions.

\begin{lemma}\label{N}
For all sufficiently small $\epsilon > 0$, there exists an $S > 0$ such that if $x(T) \in B_\epsilon(X_+)$ for $rT > S$, then $x(t) \in B_\epsilon(X_+)$ for all $t \ge T$.
\end{lemma}
\begin{proof}
Let us assume for the sake of simplicity that $X_+ = 0$ and $\lambda_+ = 0$. Since $(x,\lambda) = (0,0)$ is attracting, all eigenvalues of $A = \partial_x f(0,0)$ have negative real part, so there is some $k > 0$ such that $\Re(\mu) < -k$ for any eigenvalue $\mu$ of $A$. We can choose an inner product $\langle , \rangle$ on $U$ such that $\langle Ax, x \rangle \le -k \langle x,x \rangle$ for all $x \in U$. This defines a norm $||x|| = \langle x,x \rangle^{1/2}$. \ignore{\color{red} (See Chapter 5 of \cite{hirsch}.)} By Taylor's formula in several variables we can write
\begin{align*}
f(x,\lambda) 
\ignore{& = f(0,0) + \partial_x f(0,0)x + \partial_\lambda f(0,0)\lambda \\
& + \frac{1}{2} \partial_{x^2} f(0,0)x^2 + \partial_{x \lambda} f(0,0)x\lambda + \frac{1}{2} \partial_{\lambda^2} f(0,0)\lambda^2 \\
& + \cdots \\}
& = Ax + \alpha(x,\lambda) + \beta(x),
\end{align*}
where $||\alpha(x,\lambda)|| \le \gamma(x,\lambda)|\lambda|$ for a positive continuous $\gamma$, and $||\beta(x)|| \le \delta(x)||x||$, where $\delta$ is positive, continuous and $\delta(x) \to 0$ as $x \to 0$. Then we can write \eqref{augmented} as
\begin{align*}
\frac{dx}{dt} & = Ax + \alpha(x,\Lambda(s)) + \beta(x) \\
\frac{ds}{dt} & = r
\end{align*}

For a given $\epsilon > 0$ and $S > 0$, define $N_{\epsilon, S} = B_\epsilon(0) \times [S,\infty)$, where $B_\epsilon(0) = \{x \in U: ||x || < \epsilon\}$. Note that
\begin{align*}
\frac{d}{dt} \left(||x||^2\right) & = 2 \langle \dot x,x \rangle \\
& = 2\langle Ax,x \rangle + 2\langle \alpha(x,\Lambda(s)),x\rangle + 2 \langle \beta(x), x \rangle \\
& \le -2 k \langle x,x \rangle + 2 ||\alpha(x,\Lambda(s))|| \cdot ||x|| + 2 ||\beta(x)|| \cdot ||x|| \\
& \le -2 k ||x||^2 + 2\gamma(x,\Lambda(s))|\Lambda(s)| \cdot ||x|| + 2 \delta(x) \cdot ||x||^2 \\
& = 2 ||x||^2 \left(-k + 2\gamma(x,\Lambda(s)) \frac{|\Lambda(s)|}{||x||} + \delta(x) \right)
\end{align*}
Choose $\epsilon > 0$ such that if $||x|| \le \epsilon$, $\delta(x) < \frac{k}{2}$. Then choose $S > 0$ such that if $s \ge S$, $|\Lambda(s)| < \frac{k \epsilon}{4M}$ where $M > \sup_{||x|| \le \epsilon, \lambda \in [\lambda_-,\lambda_+]} \gamma(x,\lambda)$. Thus, if $||x|| = \epsilon$ and $s \ge S$,
\begin{align*}
\frac{d}{dt} \left(||x||\right)^2 \ignore{& \le 2 ||x||^2 \left( -k + 2\gamma(x,\Lambda(s)) \frac{|\Lambda(s)|}{||x||} + \delta(x) \right) \\}
& < 2 \epsilon^2 \left(-k + 2\gamma(x,\Lambda(s)) \left(\frac{k \epsilon}{4M\epsilon}\right) + \frac{k}{2} \right) \\
& < 2 \epsilon^2 \left( -k + \frac{k}{2} + \frac{k}{2} \right) \\
& = 0
\end{align*}
Therefore, for sufficiently small $\epsilon > 0$ there exists an $S>0$ such that the vector field of \eqref{augmented} points into $N_{\epsilon, S}$, so $N_{\epsilon, S}$ is forward invariant.
\end{proof}

If $p$ is an attracting fixed point of an autonomous system, there is a neighborhood $V$ of $p$ such that all trajectories with initial conditions in $V$ converge to $p$. This next lemma shows that a similar thing is true for $X_+$ in \eqref{augmented}, where the attracting neighborhood around $X_+$ extends both in the $x$- and $s$-dimensions.

\begin{lemma}\label{basinAttraction}
There exists an $\epsilon > 0$ and an $S > 0$ such that if $|x(t) - X_+| < \epsilon$ for $rt > S$, then $x(t) \to X_+$ as $t \to \infty$.
\end{lemma}
\begin{proof}
Pick an $\epsilon > 0$ sufficiently small for Lemma \ref{N}. Make $\epsilon$ smaller if necessary so that $\overline{B_\epsilon(X_+)} \subset \B(X_+,\lambda_+)$. Then by Lemma \ref{N}, there exists an $S > 0$ such that if $x(t) \in B_\epsilon(X_+)$ for $rT > S$, then $x(t) \in B_\epsilon(X_+)$ for all $t \ge T$. Now fix $r > 0$. Since $\overline{B_\epsilon(X_+)}$ is compact, there is some $y \in \overline{B_\epsilon(X_+)}$ such that $x(t_n) \to y$ as $t_n \to \infty$. But $y \in \B(X_+,\lambda_+)$ by assumption, so $y \cdot t \to X_+$ in the autonomous system \eqref{autonomous}. Therefore, by Lemma \ref{subsequence}, there exists a $\{u_n\} \to \infty$ such that $x(u_n) \to X_+$. \par

Now pick any $\delta \in (0,\epsilon)$. Then by Lemma \ref{N}, there exists some $S_\delta > 0$ such that if $x(T) \in B_\delta(X_+)$ for $rT > S_\delta$, then $x(T) \in B_\delta(X_+)$ for all $t \ge T$. By the previous paragraph, there is a $u_{n_\delta} > S_\delta/r$ such that $|x(u_{n_\delta}) - X_+| < \delta$. Therefore, $|x(t) - X_+| < \delta$ for all $t \ge u_{n_\delta}$. Hence $x(t) \to X_+$ as $t \to \infty$.
\end{proof}

Finally, we prove one last lemma showing that if a trajectory $x(t)$ is in a compact subset of $\B(X_+,\lambda_+)$ for large enough $t$, then $x(t)$ converges to $X_+$.

\begin{lemma}\label{compactBasinSubset}
Let $K \subset \B(X_+,\lambda_+)$ be compact. Then there exists an $S > 0$ such that if $x(T) \in K$ for $rT > S$, then $x(t) \to X_+$ as $t \to \infty$.
\end{lemma}
\begin{proof}
By Lemma \ref{basinAttraction}, there is an $\epsilon > 0$ and an $S_1 > 0$ such that if $|x(t) - X_+| < \epsilon$ for $rt > S_1$, then $x(t) \to X_+$ as $t \to \infty$. Since $K \subset \B(X_+,\lambda_+)$ is compact, there is some $T_0 > 0$ such that $y \cdot_{\lambda_+} t \in B_{\epsilon/2}(X_+)$ for any $y \in K$ and $t \ge T_0$. Also, there is some $S_2 > 0$ such that if $x(T) = y_0 \in K$ for $rT > S_2$, then $|x(T + T_0) - y_0 \cdot_{\lambda_+} T_0| < \epsilon/2$ for any $y_0 \in K$. \par

Take $S = \max\{S_1,S_2\}$. Then, suppose $x(T) \in K$ for $rT > S$. If $x(T) = y_0$, then
\begin{align*}
|x(T + T_0) - X_+| & \le |x(T + T_0) - y_0 \cdot_{\lambda_+} T_0| + |y_0 \cdot_{\lambda_+} T_0 - X_+| \\
& < \epsilon/2 + \epsilon/2 \\
& = \epsilon
\end{align*}
Therefore, $x(t) \to X_+$ as $t \to \infty$.
\end{proof}

Now we are ready to prove the generalization of statements 2 and 3 of Theorem \ref{1D}:

\begin{thm}\label{highDTheorem}
Suppose we have a system of the form \eqref{augmented} for any $n \in \N$. Let $(X(s),\Lambda(s))$ be a stable path. Set $X_\pm = \lim_{s \to \pm \infty} X(s)$.
\begin{enumerate}
\item If there is another stable path $(Y(s),\Lambda(s))$ with $Y_\pm = \lim_{s \to \pm \infty} Y(s)$ such that $Y_+ \ne X_+$ and there are $u < v$ such that
$$X(u) \in \B(Y(v),\Lambda(v)),$$
then $(X(s),\Lambda(s))$ is not FBS and there is a parameter shift $\tilde \Lambda$ (which is a re-scaling of $\Lambda$) such that there is R-tipping away from $X_-$ for this $\tilde \Lambda$.
\item If there is a $Y_+ \ne X_+$ such that $Y_+$ is an attracting equilibrium of \eqref{autonomous} for $\lambda = \lambda_+$ and
$$X_- \in \B(Y_+,\lambda_+),$$
then $(X(s),\Lambda(s))$ is not FBS and there is R-tipping away from $X_-$ for this $\Lambda$.
\end{enumerate}
\end{thm}
\begin{proof}
We will prove statement 1 first. Based on the assumptions, it is clear that $(X(s),\Lambda(s))$ is not forward basin stable. Pick $\epsilon > 0$ such that $K = \overline{B_\epsilon(X(u))} \subset \B(Y(v),\Lambda(v))$. By Lemma 2.3 of \cite{sebas}, there is an $r_0 > 0$ such that for all $r \in (0,r_0)$, $|x^r(s/r) - X(s)| < \epsilon/2$ for all $s \in R$. Likewise, there exists an $r_1 > 0$ such that for all $r \in (0,r_1)$, if $x^r(v/r) \in K$, then $x^r(t) \to Y_+$ as $t \to \infty$. Now fix $r \in (0,\min\{r_0,r_1\})$. \par

Following the proof of the corresponding theorem in \cite{sebas}, we will construct a reparametrization
$$\tilde \Lambda(s) := \Lambda(\sigma(s))$$
using a monotonic increasing $\sigma \in C^2(\R,\R)$ that increases rapidly from $\sigma(s) = u$ to $\sigma(s) = v$ but increases slowly otherwise. In particular, for any $M, \eta > 0$ we choose a smooth monotonic function $\sigma(s)$ such that
\begin{equation}
\begin{split}\label{sigmaFunc}
\begin{array}{rcl}
\sigma(s) = s & \text{ for } & s < u \\
1 \le \frac{d}{ds} \sigma(s) \le M & \text{ for } & u \le \sigma(s) \le u + \eta \\
\frac{d}{ds} \sigma(s) = M & \text{ for } & u + \eta < \sigma(s) < v - \eta \\
1 \le \frac{d}{ds} \sigma(s) \le M & \text{ for } & v-\eta \le \sigma(s) \le v, \text{ and } \\
\frac{d}{ds} \sigma(s) = 1 & \text{ for } & \sigma(s) > v
\end{array}
\end{split}
\end{equation}

Let $x^{[r,\tilde \Lambda]}(t)$ denote the pullback attractor to $(X_-,\lambda_-)$ with parameter change $\tilde \Lambda$. By construction, we know that $x^{[r,\tilde \Lambda]} (u/r) \in B_{\epsilon/2}(X(u))$. By choosing $M >0$ sufficiently large and $\eta > 0$ sufficiently small, we can guarantee that $x^{[r,\tilde \Lambda]} (v/r) \in B_\epsilon(X(u)) \subset K$. This guarantees that $x^{[r,\tilde \Lambda]} (t) \to Y_+$ as $t \to \infty$.

\vspace{.1in}

Now we will prove statement 2. Pick $\epsilon > 0$ such that $B_{3\epsilon}(X_-) \subset \B(Y_+,\lambda_+)$. By Lemma \ref{beginning}, there is an $S_1 > 0$ such that the pullback attractor $x^r(t)$ to $X_-$ satisfies $x^r(t) \in B_\epsilon(X_-)$ if $rt < -S_1$. By Lemma \ref{compactBasinSubset}, there is some $S_2 > 0$ such that if $x^r(t) \in \overline{B_{2\epsilon}(X_-)} \subset \B(Y_+,\lambda_+)$ for $rt > S_2$ then $x^r(t) \to Y_+$ as $t \to \infty$. Take $S = \max\{S_1,S_2\}$. \par

By continuity, there is some $M>0$ such that $|f(x,\lambda)| < M$ for all $x \in \overline{B_{2\epsilon}(X_-)}$ and $\lambda \in [\lambda_-,\lambda_+]$. Pick any
$$r > 2\frac{MS}{\epsilon}.$$
Then by the Mean Value Theorem, $x^r(S/r) \in B_{2\epsilon}(X_-)$. Therefore, $x^r(t) \to Y_+$ as $t \to \infty$ for all sufficiently large $r > 0$.
\ignore{
Suppose for the sake of contradiction that $x^r(S/r) \not \in B_{2\epsilon}(X_-)$. We know $x^r(-S/r) \in B_\epsilon(X_-)$, so let $s' = \inf\{s \in(-S,S]: x^r(s/r) \not \in B_{2\epsilon}(X_-)\}$. Then in fact $s'$ is a minimum and $s' > -S$. By the Mean Value Theorem,
\begin{align*}
|x^r(-S/r) - x^r(s'/r)| & = |f(x^r(s^*/r),\lambda)|\left|-\frac{S}{r} - \frac{s'}{r}\right| \text{ for some } s^* \in (-S,s'), \lambda \in [\lambda_-,\lambda_+] \\
& < M \frac{2S}{r} \text{ since } x^r(s^*/r) \in B_{2\epsilon}(X_-) \\
& < \epsilon
\end{align*}
But because $x^r(-S/r) \in B_\epsilon(X_-)$, this implies that $x^r(s'/r) \in B_{2\epsilon}(X_-)$, which is a contradiction. Therefore, $x^r(S/r) \in B_{2\epsilon}(X_-) \subset \B(X_+,\lambda)$. \par}
\end{proof}

\begin{ex}
\end{ex}

We can apply Theorem \ref{highDTheorem} to the Lorenz equations:

\begin{equation}
\begin{split}\label{lorenz63}
\dot x & = \sigma (y - x) \\
\dot y & = x(\rho - z) - y \\
\dot z & = xy - \beta z
\end{split}
\end{equation}

As in \cite{sparrow}, we will fix $\sigma = 10$ and $\beta = 8/3$, but we will allow $\rho$ to vary with time. The corresponding augmented system for \eqref{lorenz63} is

\begin{align*}
\dot x & = 10 (y - x) \\
\dot y & = x(\Lambda(s) - z) - y \\
\dot z & = xy - \frac{8}{3} z \\
\dot s & = r
\end{align*}
for $r > 0$ and $\rho = \Lambda(s)$ satisfying \eqref{star}. We will allow $\rho$ to monotonically increase from $15$ to $23$, so $\lambda_- = 15$ and $\lambda_+ = 23$. As explained in \cite{doedel} and \cite{sparrow}, in this parameter regime there are three equilibria, one at the origin and the other two
\begin{align*}
C_1 & = \left(\sqrt{8/3(\rho - 1)}, \sqrt{8/3(\rho - 1)}, \rho - 1\right) \\
C_2 & = \left(-\sqrt{8/3(\rho - 1)}, -\sqrt{8/3(\rho - 1)}, \rho - 1\right).
\end{align*}

Both $C_{1,2}$ are attracting, and the origin is a saddle point. There are heteroclinic connections from the origin to $C_{1,2}$, and there are periodic orbits around $C_{1,2}$. There is no chaotic attractor for these values of $\rho$, although as $\rho \nearrow \rho_{het} \approx 24.0579$, the time it takes for the unstable manifold of the origin to approach $C_{1,2}$ increases without bound. \par

We will focus on the stable path
$$C_1(s) = \left(\sqrt{8/3(\Lambda(s) - 1)}, \sqrt{8/3(\Lambda(s) - 1)}, \Lambda(s) - 1\right)$$
with $C_{1\pm} = \lim_{s \to \pm \infty} C_1(s)$ and consider the possibility of R-tipping away from $(C_{1-},\lambda_-)$. From plotting solutions to \eqref{lorenz63} in MATLAB, we see that $(C_1(s),\Lambda(s))$ is not FBS (see Figure \ref{notForwardBasinStable}). Therefore, according to Theorem \ref{highDTheorem}, we can expect R-tipping for some choices of $\Lambda$ and $r>0$. Indeed if we choose
$$\Lambda(s) = 4\tanh(s) + 19$$
then for some values of $r > 0$ the pullback attractor to $(C_{1-},\lambda_-)$ tracks $(C_1(s),\Lambda(s))$ and for some values of $r > 0$ it tips to $(C_2(s),\Lambda(s))$ (see Figure \ref{fig:Lorenz}).

\begin{figure}
\centering
\includegraphics[scale = 0.8]{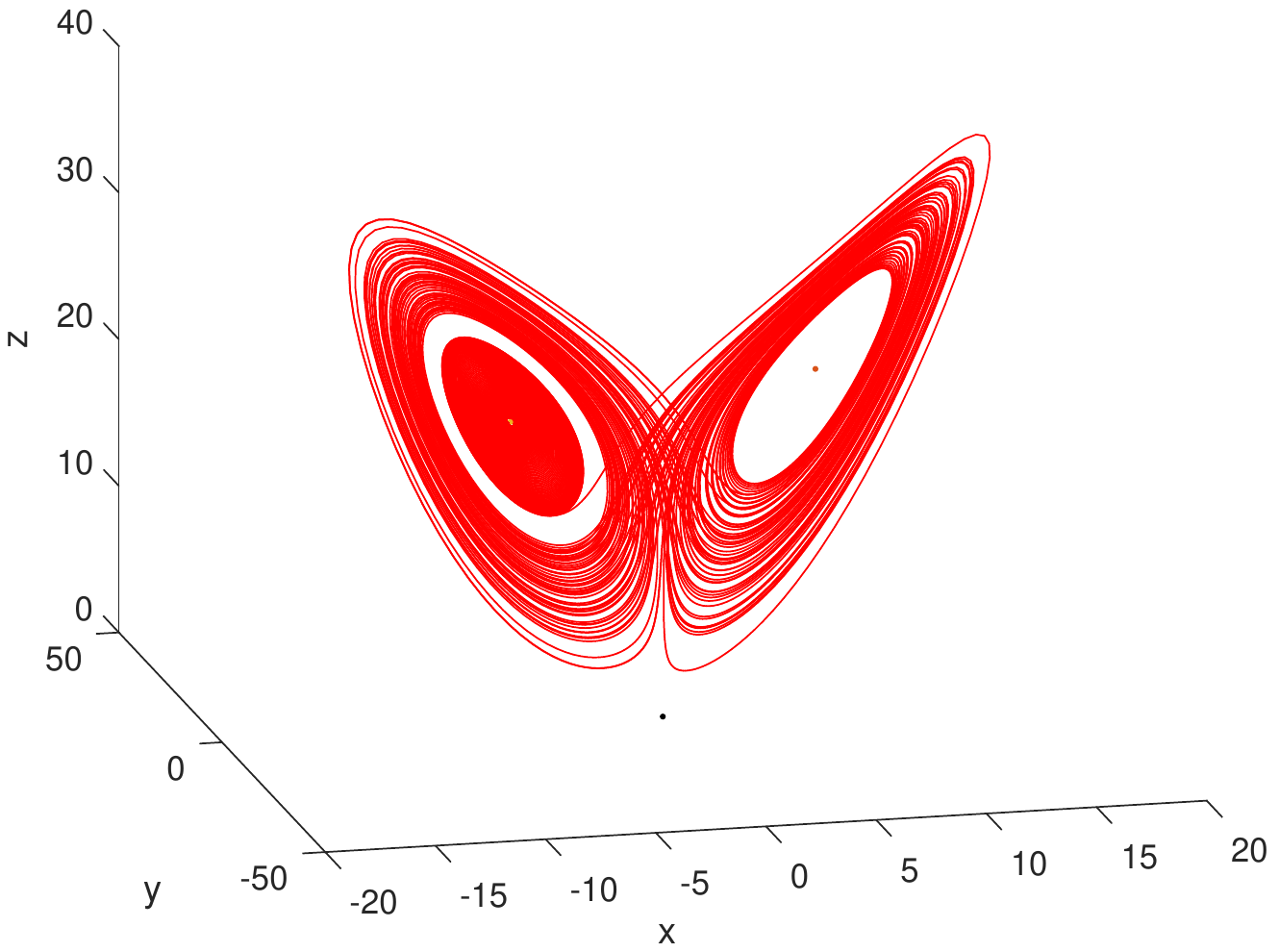}
\caption{Approximate solution curve to \eqref{lorenz63} with $\rho = 22.9$ and initial condition $\left(\sqrt{8/3(14.1)},\sqrt{8/3(14.1)},14.1 \right)$ (i.e. $C_1(s)$ when $\Lambda(s) = 15.1$). The trajectory converges to a point on the stable path $(C_2(s),\Lambda(s))$, indicating that $(C_1(s),\Lambda(s))$ is not forward basin stable.}
\label{notForwardBasinStable}
\end{figure}

\begin{figure}
	\centering
	\begin{subfigure}[b]{0.45\textwidth}
		\includegraphics[width = \textwidth]{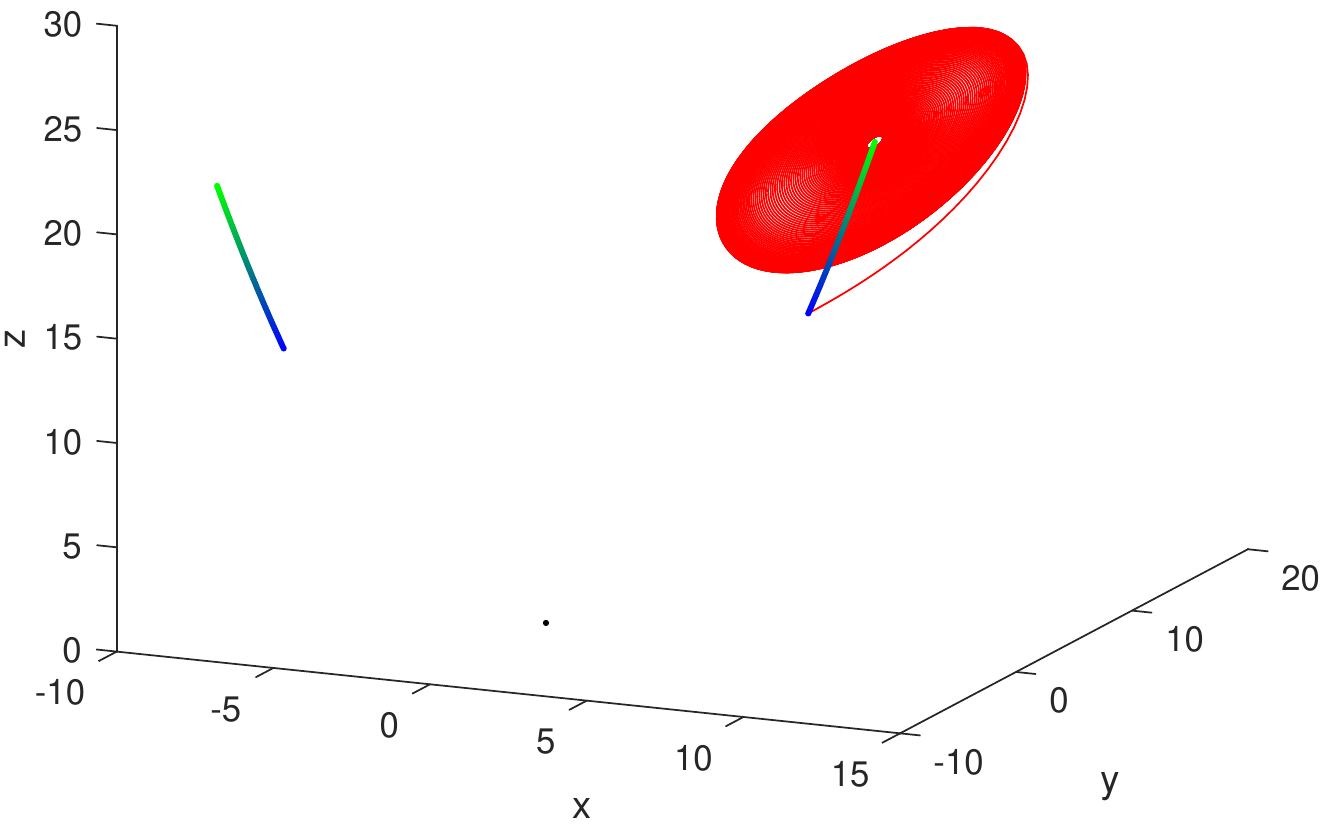}
		\caption{r = 13}
		\label{ex:LorenzNotTips}
	\end{subfigure}
	\begin{subfigure}[b]{0.45\textwidth}
		\includegraphics[width = \textwidth]{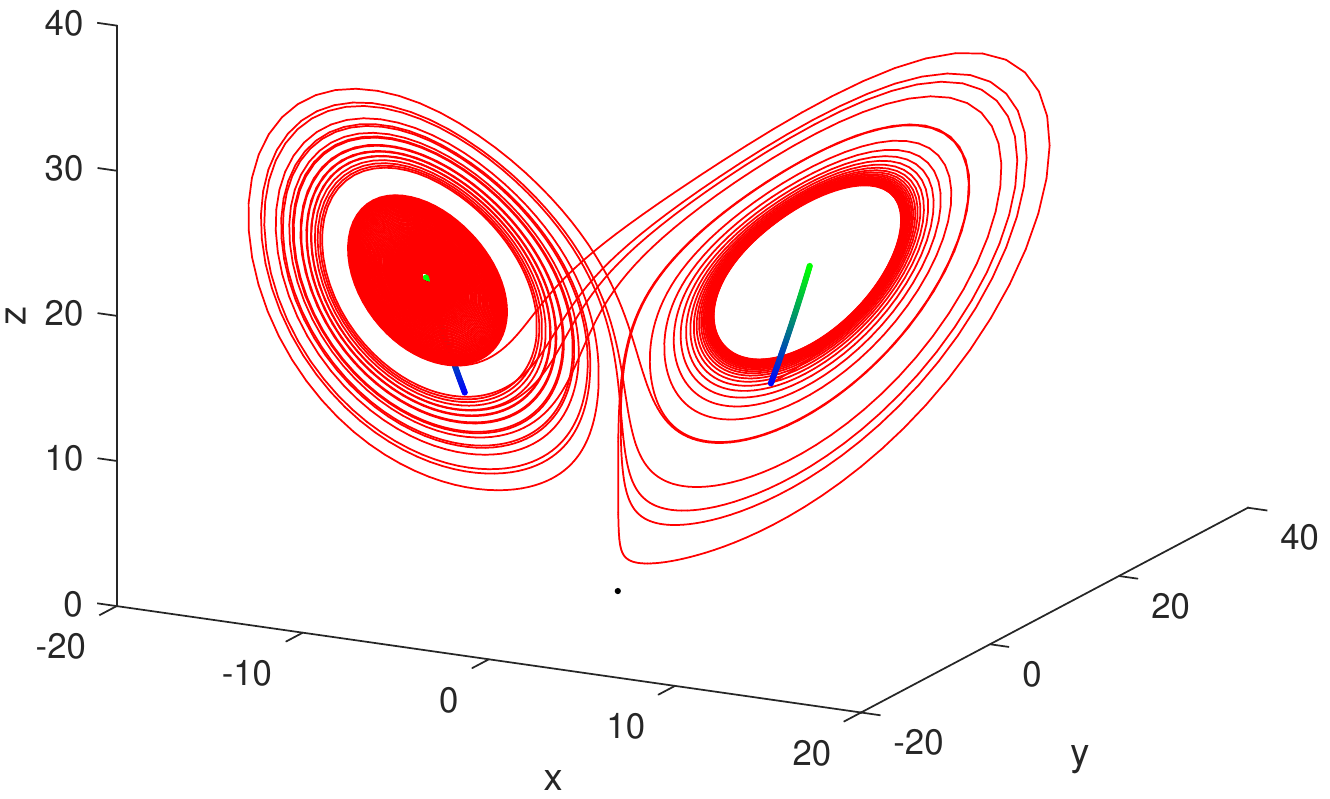}
		\caption{r = 15}
		\label{ex:LorenzTips}
	\end{subfigure}
	\caption{In both figures, the blue/green dots mark the positions of $(C_{1,2}(s),\Lambda(s))$. The blue dots correspond to small values of $s$, and the green dots correspond to large values of $s$. The red curve is the pullback attractor to $(C_{1-},\lambda_-)$. When $r = 13$, the trajectory endpoint tracks $(C_1(s),\Lambda(s))$, but when $r = 15$ it does not.}
	\label{fig:Lorenz}
\end{figure}

\section{Forward Basin Stability and Forward Inflowing Stability}\label{forwardBasinStable}

Now that we have successfully generalized statements 2 and 3 of Theorem \ref{1D}, we will turn our attention to statement 1, which says that if a path is FBS in a 1-dimensional system, then there will be no R-tipping away from that path. However, as the next example shows, FBS is not enough to prevent R-tipping in systems where $n > 1$.

\newpage

\begin{ex}\label{counterex}
\end{ex}

Consider the following 2-dimensional system (which we have adapted from Example 5.11 of \cite{Robinson}):

\begin{equation}
\begin{split}\label{homolambda}
\dot x & = -y \\
\dot y & = -(x-\lambda)+2(x-\lambda)^3 - y((x-\lambda)^2 - (x-\lambda)^4 - y^2)
\end{split}
\end{equation}

Then \eqref{homolambda} has fixed points at $(\lambda,0)$ and $\left( \lambda \pm \frac{1}{\sqrt{2}}, 0 \right)$. There are two homoclinic orbits at $(\lambda,0)$. Both $\left(\lambda \pm \frac{1}{\sqrt{2}}, 0 \right)$ are attracting, and their basins of attraction are the regions inside the corresponding homoclinic orbits. See Figure \ref{fig:homo}.

\begin{figure}
\centering
\includegraphics[scale = 0.5]{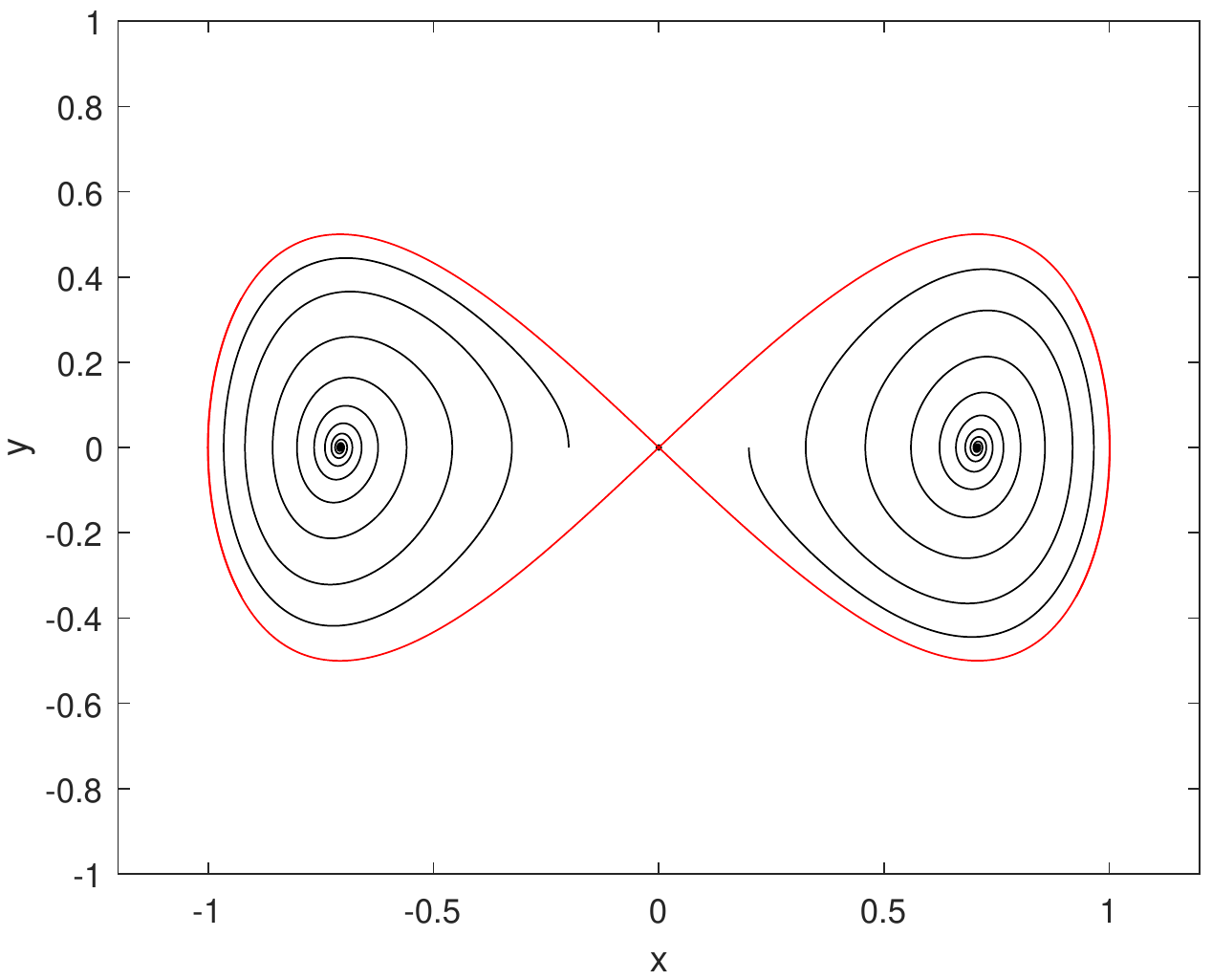}
\caption{Phase portrait for system \eqref{homolambda}. $(\lambda,0)$ is a saddle point with two homoclinic orbits (shown in red). Both $\left(\lambda \pm \frac{1}{\sqrt{2}}, 0 \right)$ are attracting equilibria; their basins of attraction are the regions inside the homoclinic loops.}
\label{fig:homo}
\end{figure}

Then we will let $\lambda$ change with time at a rate $r > 0$ by setting $\lambda = \Lambda(s)$ and $s=rt$:

\begin{equation}
\begin{split}\label{homoaug}
\dot x & = -y \\
\dot y & = -(x-\Lambda(s))+2(x-\Lambda(s))^3 - y((x-\Lambda(s))^2 - (x-\Lambda(s))^4 - y^2) \\
\dot s & = r
\end{split}
\end{equation}

For $\Lambda$ we will take $\Lambda(s) = \frac{13}{40} (1 + \tanh(s))$ so that $\lambda_- = 0$ and $\lambda_+ = 0.65 < \frac{1}{\sqrt{2}}$. Let $X(s) = \left(\Lambda(s) + \frac{1}{\sqrt{2}}, 0 \right)$. Then $X_- = \left( \frac{1}{\sqrt{2}}, 0 \right)$ and $X_+ = \left( \frac{1}{\sqrt{2}} + \frac{13}{20}, 0 \right)$. Because $0 < \Lambda(s) < \frac{1}{\sqrt{2}}$ for all $s$, the stable path $(X(s),\Lambda(s))$ is FBS. Nevertheless, R-tipping can occur away from $X_-$. See Figure \ref{fig:homotip}.

\begin{figure}
	\centering
	\begin{subfigure}[b]{0.45\textwidth}
		\includegraphics[width = \textwidth]{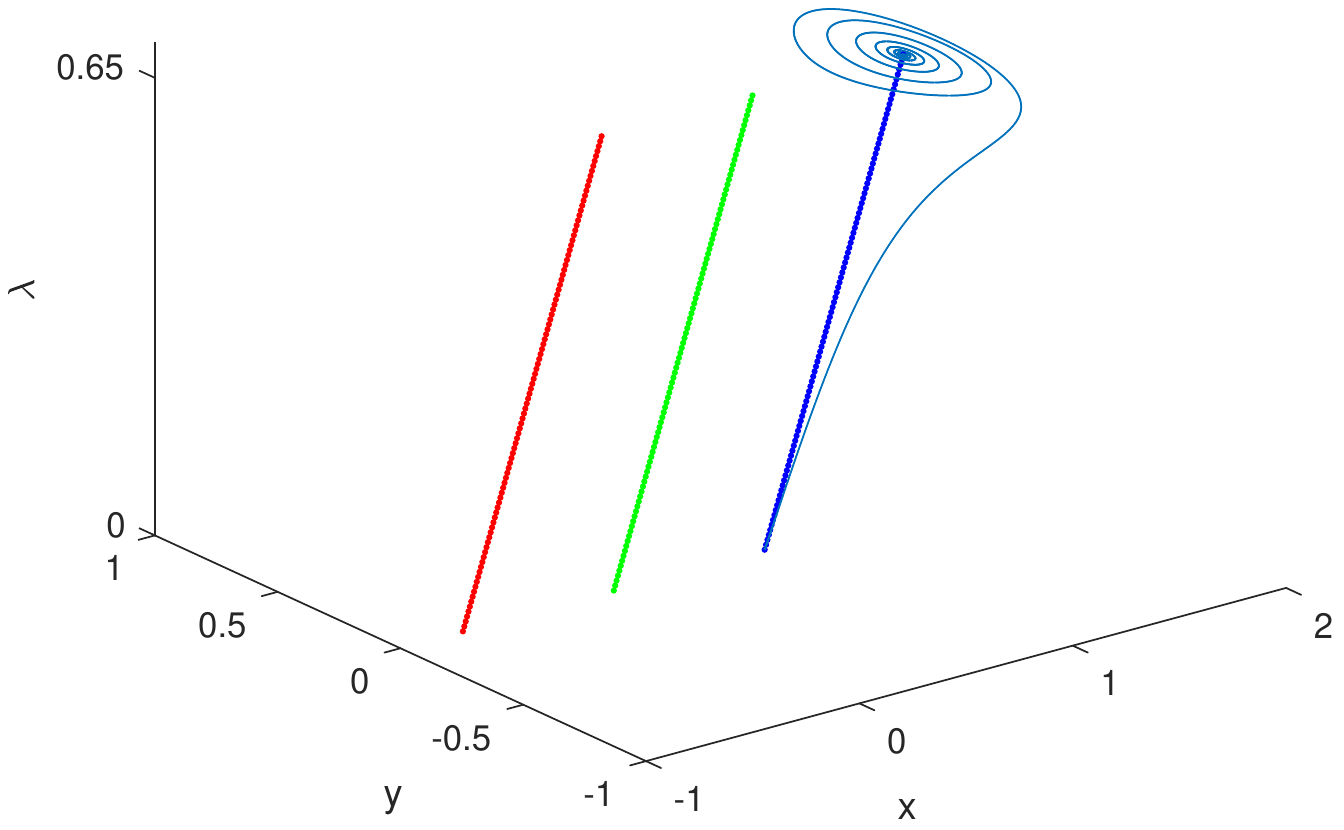}
		\caption{$r = 1$}
		\label{fig:tracks}
	\end{subfigure}
	\begin{subfigure}[b]{0.45\textwidth}
		\includegraphics[width = \textwidth]{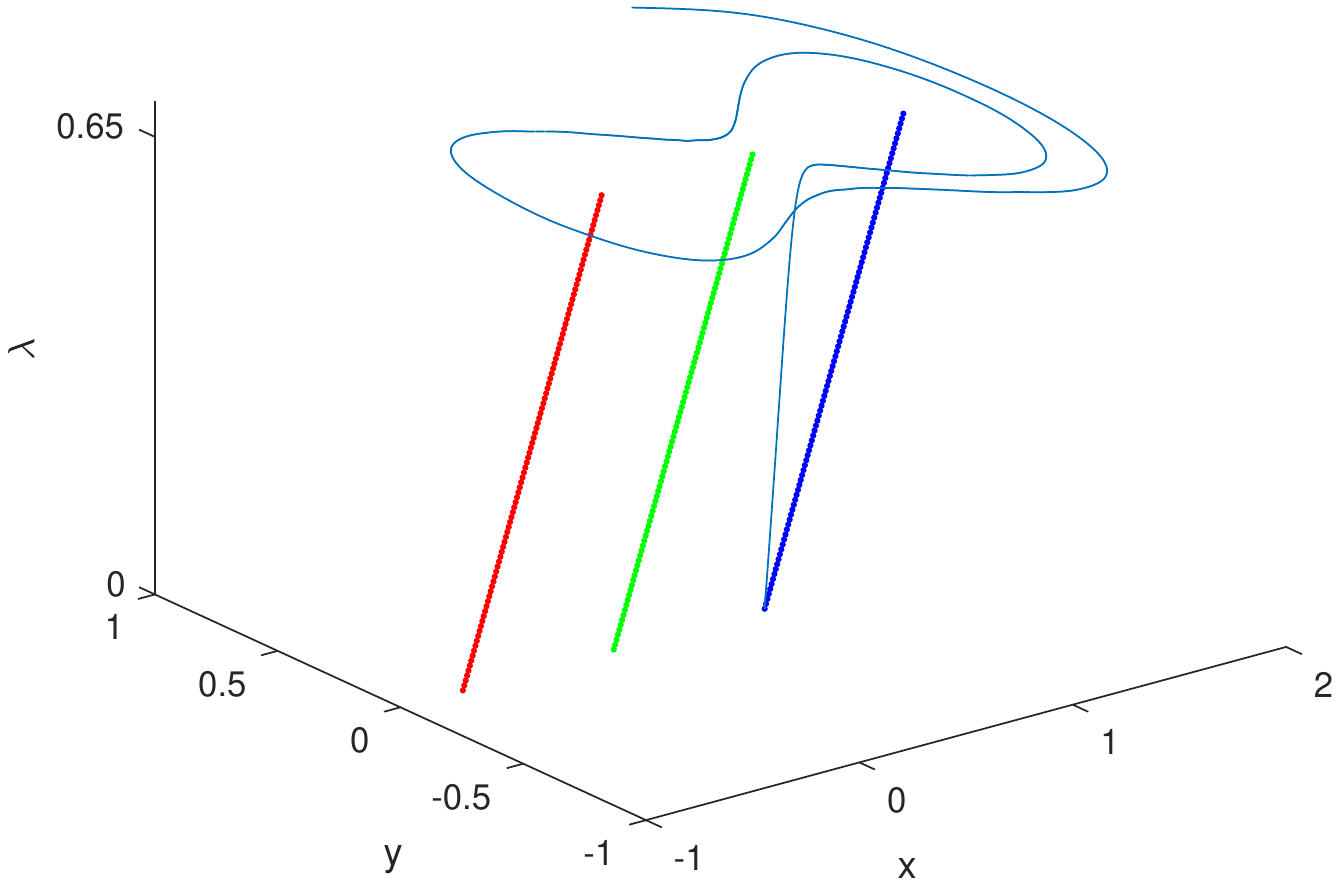}
		\caption{$r = 5$}
		\label{fig:nottracks}
	\end{subfigure}
	\caption{In both figures, the red dots represent the stable path $\left(-1/\sqrt{2} + \Lambda(s),0\right)$, the green dots represent the unstable path $(\Lambda(s),0)$, and the blue dots represent the stable path $X(s) = \left(1/\sqrt{2} + \Lambda(s),0\right)$. The blue curve is the pullback attractor to $X_-$. When $r = 1$, the pullback attractor endpoint tracks $X(s)$. When $r = 5$, and the pullback attractor diverges to infinity (does not endpoint track $X(s)$). Therefore, R-tipping has occurred. This shows that R-tipping can occur even when a path is forward basin stable in multi-dimensional systems.}
	\label{fig:homotip}
\end{figure}

\vspace{.2in}

Example \ref{counterex} shows that FBS of a path is not enough to guarantee against R-tipping in 2-dimensional systems. The reason that FBS is not sufficient in a 2-dimensional (or higher) system is that a point $x$ might be in the basin of attraction of a fixed point $p$, but the velocity vector at $x$ may not point toward $p$. The more dimensions there are in a system, the more directions there are to move, so in a sense this makes R-tipping more likely to happen. Although Example \ref{counterex} is an example of a 2-dimensional system, it would not be difficult to construct a system of higher dimension in which there can be R-tipping away from a path that is FBS. \par

Therefore, since FBS is not enough to prevent R-tipping in systems of dimension greater than 1, we want to find a different condition that is sufficient to prevent R-tipping. We propose a condition, called {\em forward inflowing stability} (FIS), which guarantees that R-tipping cannot happen away from a stable path. In what follows, we will assume that we have a system of the form \eqref{augmented} with a stable path $(X(s),\Lambda(s))$.

\begin{defn}\label{inflowingstable}
We say the stable path $(X(s),\Lambda(s))$ is {\em forward inflowing stable} if for each $s \in \R$ there is a compact set $K(s)$ such that
\begin{enumerate}
\item $X(s) \in \Int K(s)$ for all $s \in \R$;
\item if $s_1 < s_2$, then $K(s_1) \subset K(s_2)$;
\item if $x \in \partial K(s)$, then $f(x,\Lambda(s))$ points strictly into $K(s)$;
\item $X_\pm \in \Int K_\pm$ where $K_-  = \bigcap_{s \in \R} K(s)$ and $K_+ = \overline{\bigcup_{s \in \R} K(s)}$; and
\item $K_+ \subset \B(X_+,\lambda_+)$ is compact.
\end{enumerate}
\end{defn}

Just as the notion of FBS compares the positions of equilibria along a path to basins of attraction later on in the path, FIS compares the positions of equilibria along the path to forward invariant sets (sets for which solutions ``flow in'') later on down the path.

\begin{prop}\label{inflowing}
If the stable path $(X(s),\Lambda(s))$ is FIS, then there is no R-tipping away from $X_-$.
\end{prop}
\begin{proof}
Fix $r > 0$. By FIS, there exist sets $K(s)$ satisfying the requirements of Definition \ref{inflowingstable}. Set $K = \cup_{s \in \R} K(s) \times \{s\}$. If we pick a point on the boundary of $K$ when $s = s_0$, then $\frac{dx}{dt} = f(x,\Lambda(s_0))$ points strictly into $K(s_0)$. Since $K(s_0) \subset K(s)$ if $s_0 \le s$ and $\frac{ds}{dt} = r > 0$, this implies that the vector field of \eqref{augmented} points strictly into $K$. Therefore, $K$ is forward invariant under the flow of \eqref{augmented}. \par

Let $x^r(t)$ be the pullback attractor to $X_-$. Because $X_- \in \Int K_-$ and $K_- = \bigcap_{s \in \R} K(s)$, there is a $T \in \R$ such that $x^r(t) \in K(rt)$ for all $t < T$. Since $K$ is forward invariant, this implies that $x^r(t) \in K(rt)$ for all $t \in \R$. In particular, $x^r(t) \in K_+$ for all $t \in \R$. \par

We know $K_+ \subset \B(X_+,\lambda_+)$ is compact. By Lemma \ref{compactBasinSubset} this implies $x^r(t) \to (X_+,\lambda_+)$ as $t \to \infty$. Therefore, $x^r(t)$ endpoint tracks the stable branch $(X(s), \Lambda(s))$ regardless of $r > 0$, so there is no R-tipping.
\end{proof}

\begin{ex}\label{1Dexample}
\end{ex}
Consider Figure \ref{fig:1dISex}. We will assume that $\Lambda(s)$ is injective, so that $s$ and $\lambda$ are in one-to-one correspondence. Let $(X(s),\Lambda(s))$ be the stable path that is defined for all $\lambda$-values. The red curves specify a choice of $K(s)$ in the following way: $K(s)$ is the closed interval between the two red curves when $\lambda = \Lambda(s)$. Based on what is shown in the figure, $\{K(s)\}$ satisfies the requirements in Definition \ref{inflowingstable}, which shows that $(X(s),\Lambda(s))$ is FIS. The set $K = \cup_{s \in \R} K(s) \times \{s\}$ forms a forward invariant``tube'' around the stable path $(X(s),\Lambda(s))$. As shown in Proposition \ref{inflowing}, the pullback attractor for $X_-$ is always contained in $K$. There can be no R-tipping away from $X_-$ for this reason. \par

\begin{figure}
\centering
\includegraphics[scale = 0.7]{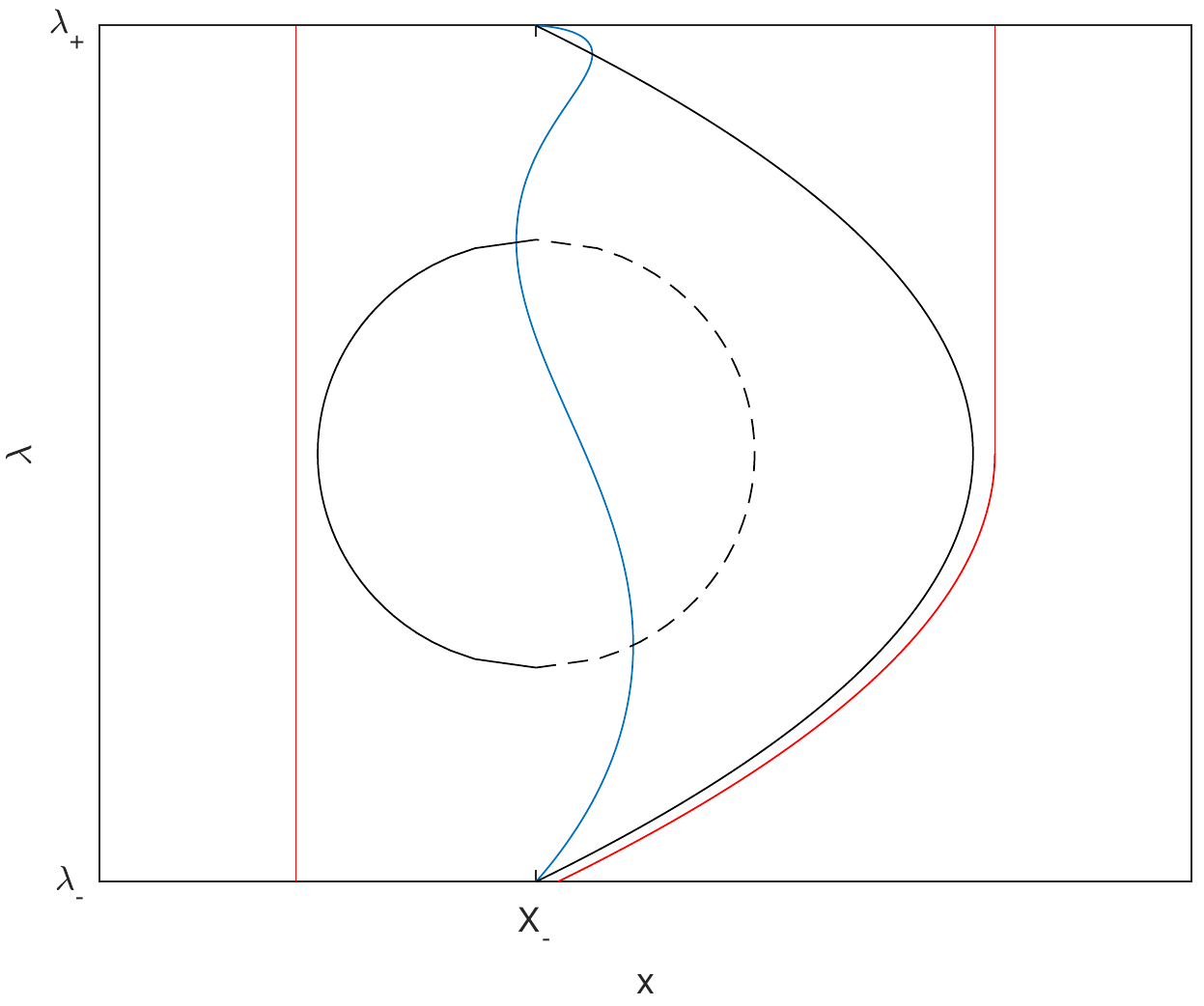}
\caption{Here there three paths: two stable (solid black lines) and one unstable (dashed black line). Let $(X(s),\Lambda(s))$ be the stable path that is defined for all $\lambda$-values. The red curves define $K(s)$ for each $s$ if $\Lambda(s)$ is one-to-one. The pullback attractor for $(X_-,\lambda_-)$ is shown in blue. Notice that the pullback attractor is fully contained in $K = \cup_{s \in \R} K(s) \times \{s\}$ and hence endpoint tracks $(X(s),\Lambda(s))$.}
\label{fig:1dISex}
\end{figure}

\vspace{.2in}

In general, FBS and FIS are conditions that are independent of each other. The path in Example \ref{1Dexample} is not FBS but is FIS. Hence, FIS does not imply FBS. Likewise, FBS does not imply FIS, as shown in Figure \ref{fig:1dFBSex}. Also note that in multi-dimensional systems FBS cannot imply FIS, as FIS prevents R-tipping, but FBS does not.

\begin{figure}
\centering
\includegraphics[scale = 0.7]{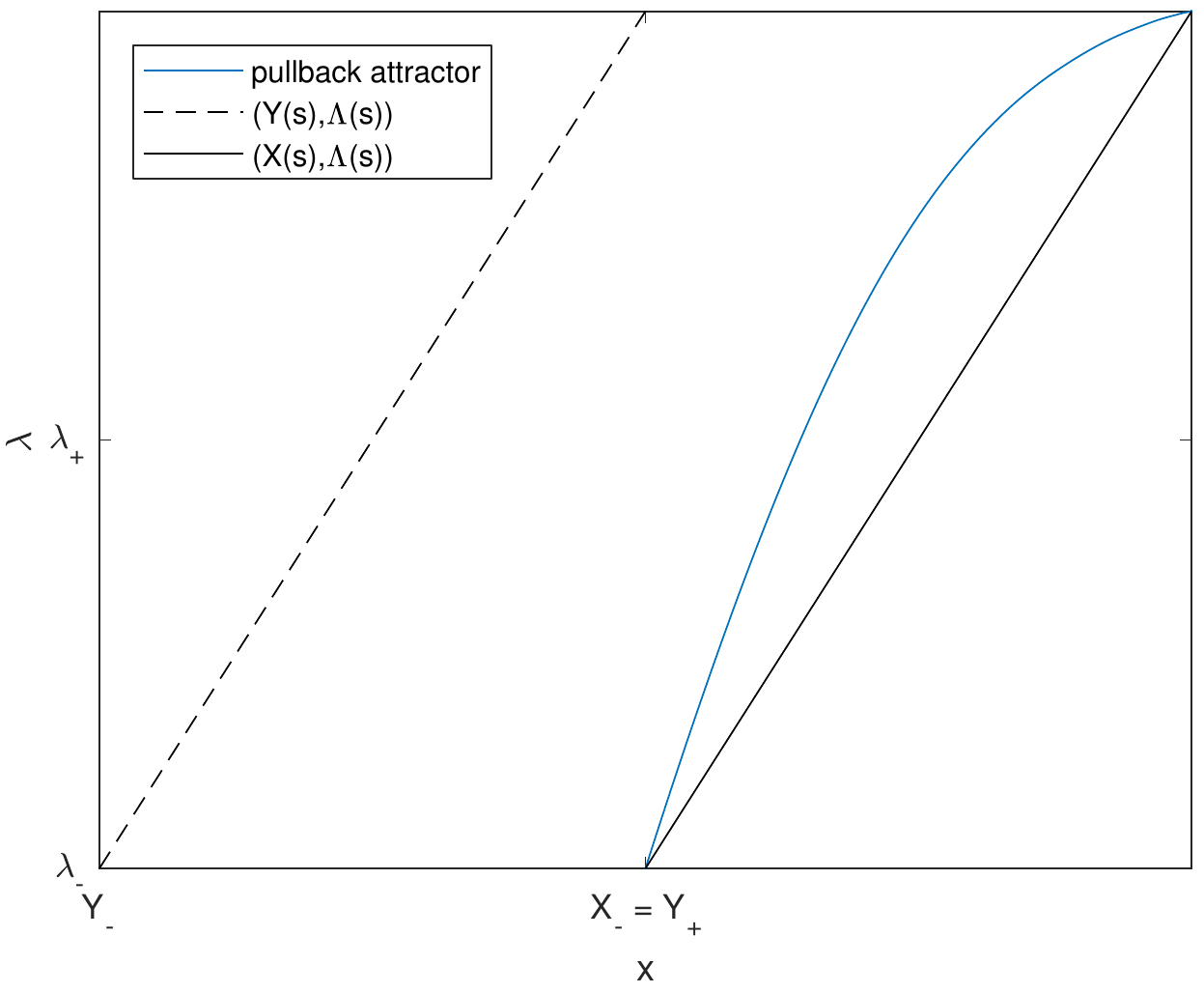}
\caption{The solid black curve is the stable path $(X(s),\Lambda(s))$, and the dashed black curve is an unstable path $(Y(s),\Lambda(s))$ satisfying $X_- = Y_+$. The blue curve is the pullback attractor to $X_-$. Assuming that $\Lambda(s)$ is one-to-one, $(X(s),\Lambda(s))$ is forward basin stable. However, it is not forward inflowing stable, since $X_-$ is on the boundary of $\B(X_+,\lambda_+)$. Any possible choice of $K_-$ must contain a neighborhood of $X_-$. Since $K_- \subset K_+$, $K_+$ cannot be fully contained in $\B(X_+,\lambda_+)$.}
\label{fig:1dFBSex}
\end{figure}

\section{Monotone Systems}\label{monotoneSection}

We will now focus our attention on rate-induced tipping in a special class of systems called {\em monotone systems}. The benefit of monotone systems is that their extra structure enables us in Proposition \ref{Chris} to prove when rate-induced tipping can happen, without having to calculate the basins of attraction of the equilibria (which can be chaotic in systems of dimension 3 or more, such as in Lorenz `63--see \cite{doedel}). Likewise, in Proposition \ref{Chris2} we will be able to prove when rate-induced tipping cannot happen, using a simpler condition than inflowing stability. \par

We begin with the notation and definition of a monotone system, which are adapted from Chapter 4 of \cite{monotone}.

\begin{defn}
Suppose $x = (x_1,\ldots,x_n) ,y = (y_1,\ldots,y_n) \in \R^n$. Then
\begin{align*}
x \le y & \iff x_i \le y_i \text{ for all } i \\
x < y & \iff x \le y \text{ and } x \ne y \\
x \ll y & \iff x_i < y_i \text{ for all } i.
\end{align*}
If $K,L \subset \R^n$ are sets, then we say $K \ll L$ if $x \ll y$ for all $x \in K$ and $y \in L$.
\end{defn}

\begin{defn}
We define a system of the form
\begin{align}\label{monotone}
\dot x = f(x)
\end{align}
for $f: U \to \R^n$, $U \subset \R^n$ open, to be {\em monotone} if
$$\frac{\partial f_i}{\partial x_j} (x) \ge 0$$
for all $i \ne j$ and all $x \in U$.
\end{defn}

Monotone systems have the property that if $x \le y$ and $t \ge 0$, then $x \cdot t \le y \cdot t$. (See Chapter 4 of \cite{monotone} for a reference). We also have the following result:

\begin{lemma}\label{K1K2}
Suppose \eqref{monotone} is a monotone system. For any $p \in U$, define $K_1(p) = \{x \in U: x \le p\}$ and $K_2(p) = \{x \in U: x \ge p\}$.
\begin{enumerate}
\item If $f_i(p) < 0$ for all $i$, then the flow of \eqref{monotone} points into $K_1$ on the boundary of $K_1$.
\item If $f_i(p) > 0$ for all $i$, then the flow of \eqref{monotone} points into $K_2$ on the boundary of $K_2$.
\end{enumerate}
\end{lemma}

\begin{proof}
The result is trivial if $n = 1$, so we will assume $n \ge 2$. We will prove the statement for $K_2$; the proof for $K_1$ is similar. \par

Assume $f_i(p) > 0$ for all $i$. Pick any point $y = (y_1,\ldots,y_n)$ not equal to $p$ on the boundary of $K_2$. Then there is at least one $i$ such that $y_i = p_i$. The line $\ell$ between $y$ and $p$ can be parametrized by $t$ in the following way:
$$\ell(t) = ((y_1 - p_1) t + p_1, (y_2 - p_2) t + p_2, \ldots, (y_n - p_n) t + p_n).$$
Since $y_i = p_i$ and $p < y$, we need to show that $\dot x_i > 0$ at $y$, or $f_i(\ell(1)) > 0$. We know that $f_i(\ell(0)) > 0$, so it will suffice to show that $(f_i \circ \ell)'(t) \ge 0$. In general, we have
\begin{align*}
(f_i \circ \ell)'(t) & = D f_i(\ell(t)) \cdot \ell'(t) \\
& = \left( \frac{\partial f_i}{\partial x_1} (\ell(t)), \frac{\partial f_i}{\partial x_2} (\ell(t)), \ldots, \frac{\partial f_i}{\partial x_n} (\ell(t)) \right) \cdot (y_1 - p_1,y_2 - p_2,\ldots,y_n - p_n) \\
& \ge 0
\end{align*}
because each $\frac{\partial f_i}{\partial x_j}(\ell(t)) \cdot (y_j - p_j) \ge 0$.
\end{proof}

\begin{figure}
\centering
\includegraphics[scale = 0.7]{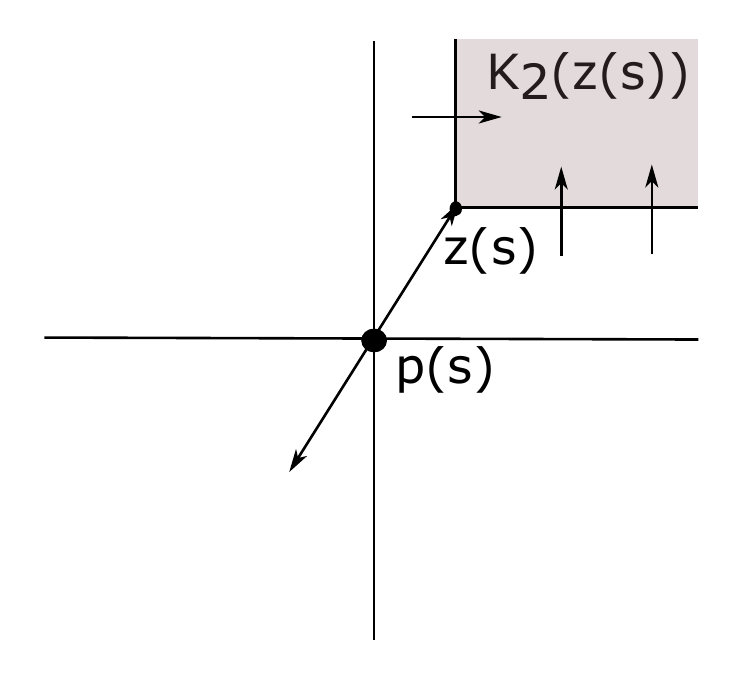}
\caption{If $Df(p(s))$ has a positive real eigenvalue whose associated eigenvector has all positive components, then there is a point $z(s)$ such that $\{z(s) \cdot_{\Lambda(s)} t\}_{t \le 0} \gg \{p(s)\}$ and $f_i(z(s),\Lambda(s)) > 0$ for all $i$. Using $z(s)$, we can define the box $K_2(z(s))$ as in Proposition \ref{Chris} satisfying $K_2(z(s)) \gg \{p(s)\}$ and such that the flow of \eqref{autonomous} with $\lambda = \Lambda(s)$ is pointing in on all sides along the boundary of $K_2(z(s))$.}
\label{fig:monotone}
\end{figure}

Then we have the following result about R-tipping in monotone systems:

\begin{prop}\label{Chris}
Suppose we have a system of the form \eqref{augmented} where $\dot x = f(x,\lambda)$ is a monotone system for each $\lambda \in [\lambda_-,\lambda_+]$. Let $(p(s),\Lambda(s))$ be a path and $(q(s),\Lambda(s))$ a stable path; denote $p_{\pm} = \lim_{s \to \pm \infty} p(s)$ and $q_{\pm} = \lim_{s \to \pm \infty} q(s)$. Suppose for all $s$ (including in the limits) $Df(p(s))$ has a positive eigenvalue whose associated eigenvector has all positive components.
\begin{enumerate}
\item If $q(s) \ll p(s)$ (resp. $q(s) \gg p(s)$) for all $s \in \R$, including in the limits as $s \to \pm \infty$, and there is a $u < v$ such that $q(u) \gg p(v)$ (resp. $q(u) \ll p(v)$), then there is a parameter shift $\tilde \Lambda$ such that there will be R-tipping away from $q_-$ for this $\tilde \Lambda$.

\item If $q(s) \ll p(s)$ (resp. $q(s) \gg p(s)$) for all $s \in \R$, including in the limits as $s \to \pm \infty$, and $q_- \gg p_+$ (resp. $q_- \ll p_+$), then there will be R-tipping away from $q_-$ for this $\Lambda$ and for large enough $r > 0$.
\end{enumerate}
\end{prop}
\begin{proof}
We will prove statement 1. The proof of statement 2 is similar but does not require any reparametrization. Suppose $q(s) \ll p(s)$ for all $s \in \R$ including in the limits and that $q(u) \gg p(v)$. Let $x^r(t)$ denote the pullback attractor to $(q_-,\lambda_-)$. \par

Pick $\epsilon > 0$ such that $\overline{B_\epsilon(q(u))} \gg \{p(v)\}$. By Lemma 2.3 of \cite{sebas}, there is an $r_0 > 0$ such that for all $0 < r < r_0$, $|x^r(s/r) - q(s)| < \epsilon/2$ for all $s \in \R$. \par

For any $s \in \R$, since $Df(p(s))$ has a positive real eigenvalue with an eigenvector that has all positive components, by the Invariant Manifold Theorem there is a point $z(s)$ such that $\{z(s) \cdot_{\Lambda(s)} t\}_{t \le 0} \gg \{p(s)\}$ and $f_i(z(s),\Lambda(s)) > 0$ for all $i$. Then define $K_2(z(s)) = \{x \in U: x \ge z(s)\}$ as in Lemma \ref{K1K2}. The flow of \eqref{autonomous} with $\lambda = \Lambda(s)$ is pointing in on all sides along the boundary of $K_2(z(s))$. See Figure \ref{fig:monotone} for an illustration. \par

Because $z(s)$ can be chosen to be arbitrarily close to $p(s)$, let us also say that $z(v)$ satisfies $\{z(v)\} \ll \overline{B_{\epsilon}(q(u))}$ and that $z(s)$ varies continuously in $s$. \par

Because the system converges as $s \to \infty$, there is an $S_0 > v$ such that the flow of the autonomous system \eqref{autonomous} with $\lambda = \Lambda(s)$ points in along the boundary of $K_2(z(S_0))$ for every $s \ge S_0$ and $q_+ \not \in K_2(z(S_0))$. Then $K_2(z(S_0)) \times [S_0,\infty)$ is forward invariant under the flow of \eqref{augmented} for any $r > 0$. Additionally, we can choose $r_1 > 0$ sufficiently small so that
$$\bigcup_{s \in [v,S_0]} K_2(z(s)) \times \{s\}$$
is forward invariant under the flow of \eqref{augmented}. Now fix $r \in (0,\min\{r_0,r_1\})$. \par

As in the proof of Theorem \ref{highDTheorem}, we will construct a reparametrization
$$\tilde \Lambda(s) := \Lambda(\sigma(s))$$
using a smooth monotonic increasing $\sigma \in C^2(\R,\R)$ that increases rapidly from $\sigma(s) = u$ to $\sigma(s) = v$ but increases slowly otherwise. In particular, for any $M, \eta > 0$ we choose a smooth monotonic function $\sigma(s)$ that satisfies \eqref{sigmaFunc}. \par

Let $x^{[r,\tilde \Lambda]}(t)$ denote the pullback attractor to $(X_-,\lambda_-)$ with parameter change $\tilde \Lambda$. By construction, we know that $x^{[r,\tilde \Lambda]} (u/r) \in B_{\epsilon/2}(q(u))$. By choosing $M >0$ sufficiently large and $\eta > 0$ sufficiently small, we can guarantee that $x^{[r,\tilde \Lambda]} (v/r) \in B_\epsilon(q(u)) \subset K_2(z(v))$. This implies that $x^{[r,\tilde \Lambda]} (t) \in K_2(z(S_0))$ for all sufficiently large $t$ and therefore $x^{[r,\tilde \Lambda]} (t) \not \to q_+$ as $t \to \infty$.
\ignore{
\vspace{.1in}

Now we will prove statement 2. Assume $q(s) \ll p(s)$ for all $s \in \R$ including in the limits and that $q_- \gg p_+$. \par

Pick $\epsilon > 0$ such that $B_{3\epsilon}(q_-) \gg \{p_+\}$. By Lemma \ref{beginning}, there is an $S_1 > 0$ such that the pullback attractor $x^r(t)$ to $q_-$ satisfies $x^r(t) \in B_\epsilon(q_-)$ if $rt < -S_1$. \par

Since $Df(p_+)$ has a positive real eigenvalue with with an eigenvector that has all positive components, by the Invariant Manifold Theorem there is a point $z$ such that $\{z \cdot_{\lambda_+} t\}_{t \le 0} \gg \{p_+\}$, $\{z\} \ll \overline{B_{2\epsilon}(q_-)}$, and $f_i(z,\lambda_+) > 0$ for all $i$. Then define $K_2(z) = \{x \in U: x \ge z\}$ as in Lemma \ref{K1K2}. \par

Pick $S_2 > 0$ such that $f_i(z, \Lambda(s)) > 0$ for all $s \ge S_2$. Then Lemma \ref{K1K2} implies that the flow of \eqref{autonomous} for $\lambda = \Lambda(s)$ with $s \ge S_2$ is pointing into $K_2(z)$ on the boundary of $K_2(z)$. Hence, $K_2(z) \times [S_2,\infty)$ is a forward invariant set under the flow of \eqref{augmented}. This implies that if $x^r(t) \in \overline{B_{2\epsilon}(q_-)} \gg \{p_+\}$ for $rt > S_2$ then $x^r(t) \gg p_+$ for all $rt \ge S_2$. \par

Take $S = \max\{S_1,S_2\}$. Then $r > 0$ can be chosen large enough so that $x^r(-S/r) \in B_\epsilon(q_-)$ implies that $x^r(S/r) \in B_{2\epsilon}(q_-)$, as in the proof of Theorem \ref{highDTheorem}. For these values of $r>0$, $x^r(t) \not \to q_+$ as $t \to \infty$. Therefore, there is R-tipping away from $q_-$ for large enough $r > 0$.}
\end{proof}

One of the benefits of Proposition \ref{Chris} is that it is not necessary to know everything about the dynamics of the autonomous system. We made no attempt to show exactly what the end behavior of the pullback attractor $x^r(t)$ to $q_-$ is; we simply needed to show that it is bounded away from $q_+$. Much of our work in proving Theorem \ref{highDTheorem} was to describe the end behavior of the pullback attractor; all that here was unnecessary because of the special monotone structure of the system. \par

Now we will see how the idea of forward inflowing stability can be applied to monotone systems to show that there will not be rate-induced tipping. As we have already seen, in multi-dimensional systems there are many different ``directions'' in which a trajectory can tip. It will be useful for the next result if we narrow our focus from tipping in general to tipping in a particular direction. We make the following definition:

\begin{defn}
Let $(q(s),\Lambda(s))$ be a stable path in system \eqref{augmented}, and let $x^r(t)$ denote the pullback attractor to $q_-$. Let $L \subset U$ be closed and $q_+ \not \in L$. We say that $x^r(t)$ {\em tips to L} if there is some $r > 0$ and $T > 0$ such that $x^r(t) \in L$ for all $t \ge T$.
\end{defn}

\begin{prop}\label{Chris2}
Suppose we have a system of the form \eqref{augmented} where $\dot x = f(x,\lambda)$ is a monotone system for each $\lambda \in [\lambda_-,\lambda_+]$. Let $(p(s),\Lambda(s))$ be a path and $(q(s),\Lambda(s))$ be a stable path; denote $p_{\pm} = \lim_{s \to \pm \infty} p(s)$ and $q_{\pm} = \lim_{s \to \pm \infty} q(s)$. Suppose for all $s \in \R$ (including in the limits) $Df(p(s))$ has a positive eigenvalue whose associated eigenvector has all positive components. If
$$q(s_1) \ll p(s_2) \text{ (resp. $q(s_1) \gg p(s_2)$)}$$
for all $s_1 \le s_2$ (including in the limits as $s_1 \to -\infty$ and $s_2 \to \infty$) then there is no R-tipping away from $q_-$ to $\{x: x \ge p_+\}$ (resp. to $\{x: x \le p_+\}$).
\end{prop}
\begin{proof}
We will assume that $q(s_1) \ll p(s_2)$ for all $s_1 \le s_2$ and prove the corresponding result. The proof of the other result is similar. \par

Because each $Df(p(s))$ has a positive real eigenvalue whose associated eigenvector has all positive components, by the Invariant Manifold Theorem there is a $z(s) \ll p(s)$ such that $z(s) \cdot_{\Lambda(s)} t \to p(s)$ as $t \to -\infty$ and $f_i(z(s),\Lambda(s)) < 0$ for all $i$. By changing $z(s)$ if necessary, we also can guarantee that $q(s) \ll z(s) \ll p(s)$ for all $s \in \R$, including in the limits and that $z(s_1) \le z(s_2)$ for all $s_1 \le s_2$. \par

Now define $K(s) = \{x \in U: x \le z(s)\}$ for all $s$, including in the limits. Then the $\{K(s)\}$ satisfy all the conditions in Definition \ref{inflowingstable} except they are not compact, and we do not know that $K_+ \subset \B(q_+,\lambda_+)$. Nevertheless, arguments like those in Proposition \ref{inflowing} show that the pullback attractor $x^r(t)$ to $q_-$ must satisfy $x^r(t) \in K_+$ for all $t \in \R$. Now $K_+ \ll \{x: x \ge p_+\}$, so $x^r(t)$ does not tip to $\{x: x \ge p_+\}$ for any $r > 0$.
\end{proof}

Notice that in Proposition \ref{Chris2} we cannot conclude that rate-induced tipping does not happen at all; it is possible that the parameter change in system \eqref{augmented} may cause rate-induced tipping to happen away from $q_-$ in another direction. But given a particular monotone system, one could perhaps apply Proposition \ref{Chris2} along with some other arguments to conclude that no rate-induced tipping is possible for a given parameter change.

\section{An Example}\label{exampleSection}

Here we give an example of a two-dimensional monotone system that will allow us to apply the things proven in this paper, particularly in Section \ref{monotoneSection}. For $x = (x_1,x_2) \in \R^2$, consider the system
\begin{equation}\label{nDexample}
\begin{split}
\dot x_1 & = -x_1^3 + ax_1^2 + bx_1 - a(b+1) + x_2 \\
\dot x_2 & = x_1-x_2
\end{split}
\end{equation}
for any $b > -1$ and $a$ satisfying $|a| <\sqrt{b+1}$. Let
$$F(x) = \begin{pmatrix}
F_1(x) \\ F_2(x)
\end{pmatrix}
= \begin{pmatrix}
-x_1^3 + ax_1^2 + bx_1 - a(b+1) + x_2 \\ x_1-x_2
\end{pmatrix}$$
denote the vector field generated by \eqref{nDexample}. The fixed points of \eqref{nDexample} are $p_1 = \left( -\sqrt{b + 1}, -\sqrt{b + 1} \right)$, $p_2 = (a,a)$, and $p_3 = \left( \sqrt{b + 1}, \sqrt{b + 1} \right)$. The derivative matrix at a point $x = (x_1,x_2)$ is
\begin{align}\label{alphaMat}
DF(x) = \begin{pmatrix}
\alpha(x) & 1 \\
1 & -1
\end{pmatrix}
\end{align}
where $\alpha(x) = -3x_1^2 + 2ax_1 + b$. This shows that \eqref{nDexample} is monotone. If $\alpha(x) < -1$, then $\eqref{alphaMat}$ has two negative eigenvalues, and if $\alpha(x) > -1$, then $\eqref{alphaMat}$ has one positive and one negative eigenvalue. In our parameter regime, $\alpha\left( -\sqrt{b + 1}, -\sqrt{b + 1} \right), \alpha\left( \sqrt{b + 1}, \sqrt{b + 1} \right) < -1$, so $p_1$ and $p_3$ are attracting equilibria, whereas $\alpha(a,a) > -1$, so $p_2$ is a saddle node. An eigenvector associated with the positive eigenvalue $\lambda_+$ of $DF(p_2)$ is
$$\begin{pmatrix}
\lambda_+ + 1 \\ 1
\end{pmatrix},$$
which points in the all-positive direction. \par

Now, let us consider the possibility of rate-induced tipping in \eqref{nDexample}. To do this, we let $a$ and $b$ depend on a parameter that can vary with time:

\begin{equation}\label{nDrtip}
\begin{split}
\dot x_1 & = -x_1^3 + a(\Lambda(s)) x_1^2 + b(\Lambda(s))x_1 - a(\Lambda(s))(b(\Lambda(s))+ 1) + x_2 \\
\dot x_2 & = x_1 - x_2 \\
\dot s & = r,
\end{split}
\end{equation}
where $\Lambda: \R \to (\lambda_-,\lambda_+)$ for some $\lambda_- < \lambda_+$ satisfies \eqref{star} and $a,b: [\lambda_-,\lambda_+] \to \R$ are smooth functions satisfying $b(\lambda) > -1$ and $|a(\lambda)| < \sqrt{b(\lambda) + 1}$ for all $\lambda \in [\lambda_-,\lambda_+]$. We will use the notation
$$F(x,\lambda) = \begin{pmatrix}
F_1(x,\lambda) \\ F_2(x,\lambda)
\end{pmatrix}
= \begin{pmatrix}
-x_1^3 + a(\lambda) x_1^2 + b(\lambda)x_1 - a(\lambda)(b(\lambda)+ 1) + x_2 \\ x_1 - x_2
\end{pmatrix}$$
to denote the first two components of the vector field of \eqref{nDrtip}. There are two stable paths in this augmented system: $(p_1(s),\Lambda(s))$ and $(p_3(s),\Lambda(s))$. The path $(p_2(s),\Lambda(s))$ is unstable. For ease of notation, we will define
$$\lim_{s \to \pm \infty} p_i(s) = p_{i\pm}.$$

We have the following result about the possibility of R-tipping in system \eqref{nDrtip}:

\begin{prop}\label{exhaustive}
\begin{enumerate}
\item If there exist $s_1 < s_2$ such that
$$-\sqrt{b(\Lambda(s_1))+1} > a(\Lambda(s_2)) \ (\text{resp. }\sqrt{b(\Lambda(s_1))+1} < a(\Lambda(s_2)))$$
then there can be R-tipping away from $p_{1-}$ (resp. $p_{3-}$).

\item If, for all $s_1 < s_2$ (including in the limit as $s_1 \to -\infty$ or $s_2 \to \infty$),
$$-\sqrt{b(\Lambda(s_1))+1} < a(\Lambda(s_2)) \ (\text{resp. } \sqrt{b(\Lambda(s_1))+1} > a(\Lambda(s_2)))$$
then $(p_1(s),\Lambda(s))$ (resp. $(p_3(s),\Lambda(s))$) is forward inflowing stable and there can be no R-tipping away from $(p_{1-},\lambda_-)$ (resp. $(p_{3-},\lambda_-)$).
\end{enumerate}
\end{prop}
\begin{remark}
Notice that Proposition \ref{exhaustive} gives a nearly exhaustive description of whether a given parameter change will lead to R-tipping or not for \eqref{nDrtip}. The only cases left out are boundary cases when, for instance, $\sqrt{b(\Lambda(s_1)) + 1}$ is equal to, but never greater than, $a(\Lambda(s_2))$ for some $s_1 < s_2$.
\end{remark}
\begin{proof}
Statement 1 is a consequence of Proposition \ref{Chris}. The proof of statement 2 requires more work. We will show that there can be no R-tipping away from $(p_{1-},\lambda_-)$ if
$$-\sqrt{b(\Lambda(s_1))+1} < a(\Lambda(s_2))$$
for all $s_1 < s_2$. The proof of the corresponding statement is similar. \par

Let $r > 0$ and let $x^r(t)$ be the pullback attractor to $p_{1-}$. By Proposition \ref{Chris2}, $x^r(t)$ does not tip to $\{x: x \ge p_{2-}\}$, but in fact the proof of Proposition \ref{Chris2} shows something stronger: there is a $z$ satisfying $p_{1+} \ll z \ll p_{2+}$ such that $x^r(t) \in \{x: x \le z\}$ for all $t$. \par

Choose $c < \inf_{s \in \R}\left\{-\sqrt{b(\Lambda(s))+1}\right\}$. Then
$$0 < -c^3 + a(\Lambda(s)) c^2 + (b(\Lambda(s))+1) c - a(\Lambda(s))(b(\Lambda(s)) + 1)$$
for all values of $s$ (including in the limits as $s \to \pm \infty$), and therefore
$$c^3 - a(\Lambda(s)) c^2 - b(\Lambda(s)) c + a(\Lambda(s))(b(\Lambda(s)) + 1) < c.$$
Choose any $d$ satisfying
$$c^3 - a(\Lambda(s)) c^2 - b(\Lambda(s)) c + a(\Lambda(s))(b(\Lambda(s)) + 1)< d < c$$
for all values of $s$ (including in the limits as $s \to \pm \infty$) and set $p = (c, d)$. Then $p \ll p_1(s)$ for all $s$, and $F_i(p,\lambda) > 0$ for all $i$ and $\lambda \in [\lambda_-,\lambda_+]$. If we set $K(s) \equiv \{x : p \le x \le z\}$, then $\{K(s)\}$ clearly satisfies the first 4 conditions of Definition \ref{inflowingstable} to show that $(p_1(s),\Lambda(s))$ is forward inflowing stable. The only thing that remains to be shown is that $K_+ \subset \B(p_{1+},\lambda_+)$.

What we are going to show is that $K_+ = \{x: p \le x \le z\}$ is in the basin of attraction of $p_{1+}$ under the flow of \eqref{nDexample} when $\lambda = \lambda_+$.  Fix some $x \in K_+$. As shown above, $K_+$ is a forward invariant box, so $\{x \cdot t\}_{t \ge 0}$ stays in $K_+$ for all time. Hence, $\omega(x)$ is nonempty and compact. By Theorem 3.22 in Chapter 4 of \cite{monotone}, $\omega(x)$ is a fixed point, and therefore must be $p_1$. Because $x$ was an arbitrary point in $K_+$, $K_+$ is in the basin of attraction of $p_{1+}$. \par

Therefore, $(p_1(s),\Lambda(s))$ is FIS, and there cannot be R-tipping away from $(p_{1-},\lambda_-)$.
\end{proof}

Let us look at a couple of specific examples to illustrate Proposition \ref{exhaustive}.

\begin{ex}
\end{ex}
Let $\Lambda(s) = \frac{1}{2} (1 + \tanh(s))$. Then $\lambda_- = 0$ and $\lambda_+ = 1$. We define the dependence of $a$ and $b$ on the parameter $\lambda$ to be $a(\lambda) = 2\lambda$ and $b(\lambda) = 8\lambda$. This means that $b(\Lambda(s)) > -1$ and $a(\Lambda(s)) < \sqrt{b(\Lambda(s)) + 1}$ for all $s$. However,
$$\sqrt{b(\Lambda(-5)) + 1} < a(\Lambda(5)),$$
which implies by Proposition \ref{exhaustive} there can be R-tipping away from $p_{3-}$. See Figure \ref{ex:tips}. Note that because $\Lambda$ is a monotone function of $s$, we can plot the positions of the trajectory and the quasi-stable (unstable) equilibria against $\lambda = \Lambda(s)$ rather than $s$. This is convenient because the range of $s$ is infinite, but the range of $\Lambda$ is bounded.

\begin{ex}
\end{ex}
Once again, let $\Lambda(s) = \frac{1}{2} (1 + \tanh(s))$, but this time define $a(\lambda) =\frac{1}{2} \lambda$ and $b(\lambda) = \lambda$. Then $b(\Lambda(s)) > -1$ and $a(\Lambda(s)) < \sqrt{b(\Lambda(s)) + 1}$ for all $s$. Furthermore,
$$-\sqrt{b(\Lambda(s)) + 1} < -1 < 0 < a(\Lambda(s)) < \frac{1}{2} < 1 < \sqrt{b(\Lambda(s)) + 1}$$
for all $s$, so by Proposition \ref{exhaustive} there can be no R-tipping away from either $p_{1-}$ or $p_{3-}$. See Figure \ref{ex:nottips}.

\begin{figure}
	\centering
	\begin{subfigure}[b]{0.45\textwidth}
		\includegraphics[width = \textwidth]{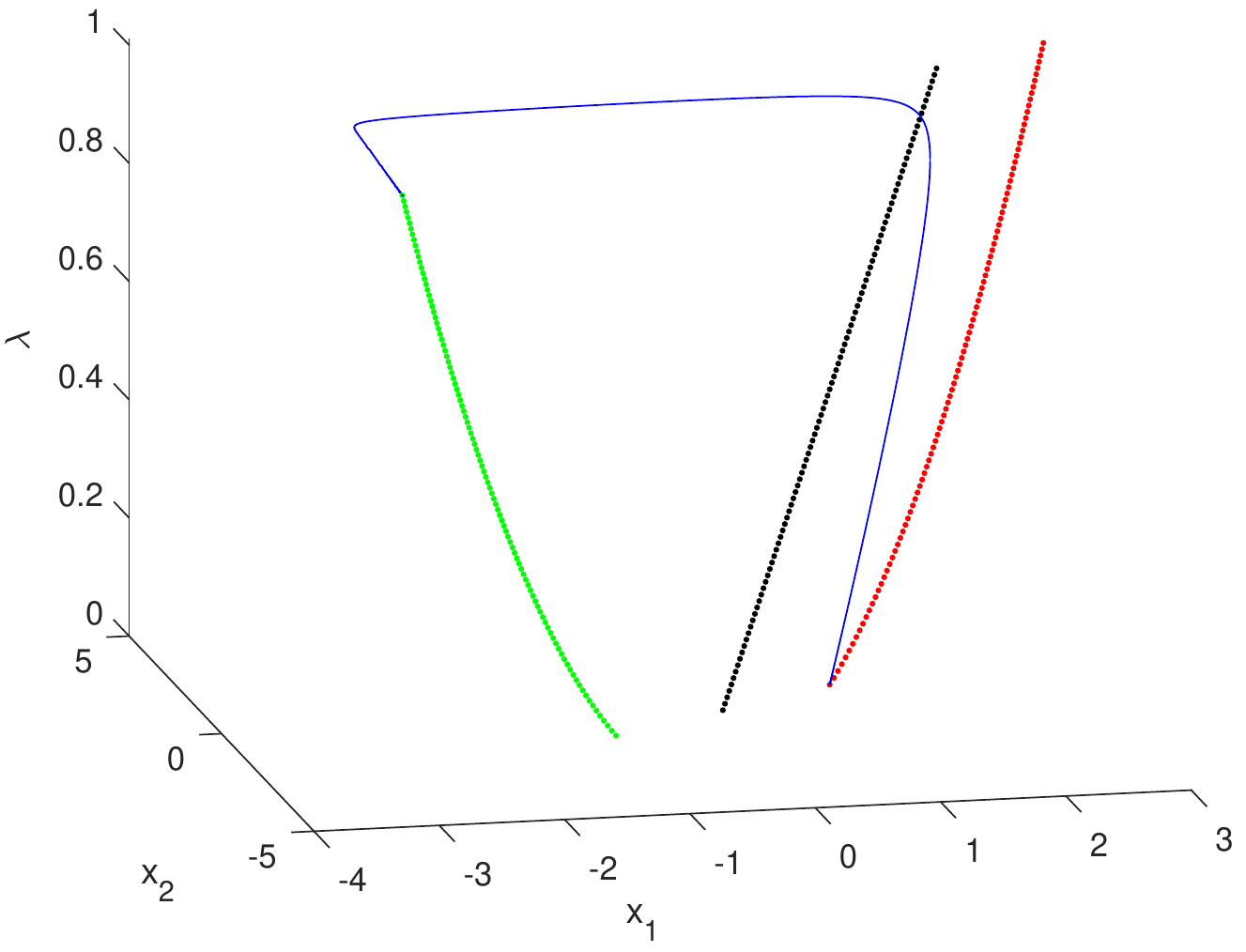}
		\caption{}
		\label{ex:tips}
	\end{subfigure}
	\begin{subfigure}[b]{0.45\textwidth}
		\includegraphics[width = \textwidth]{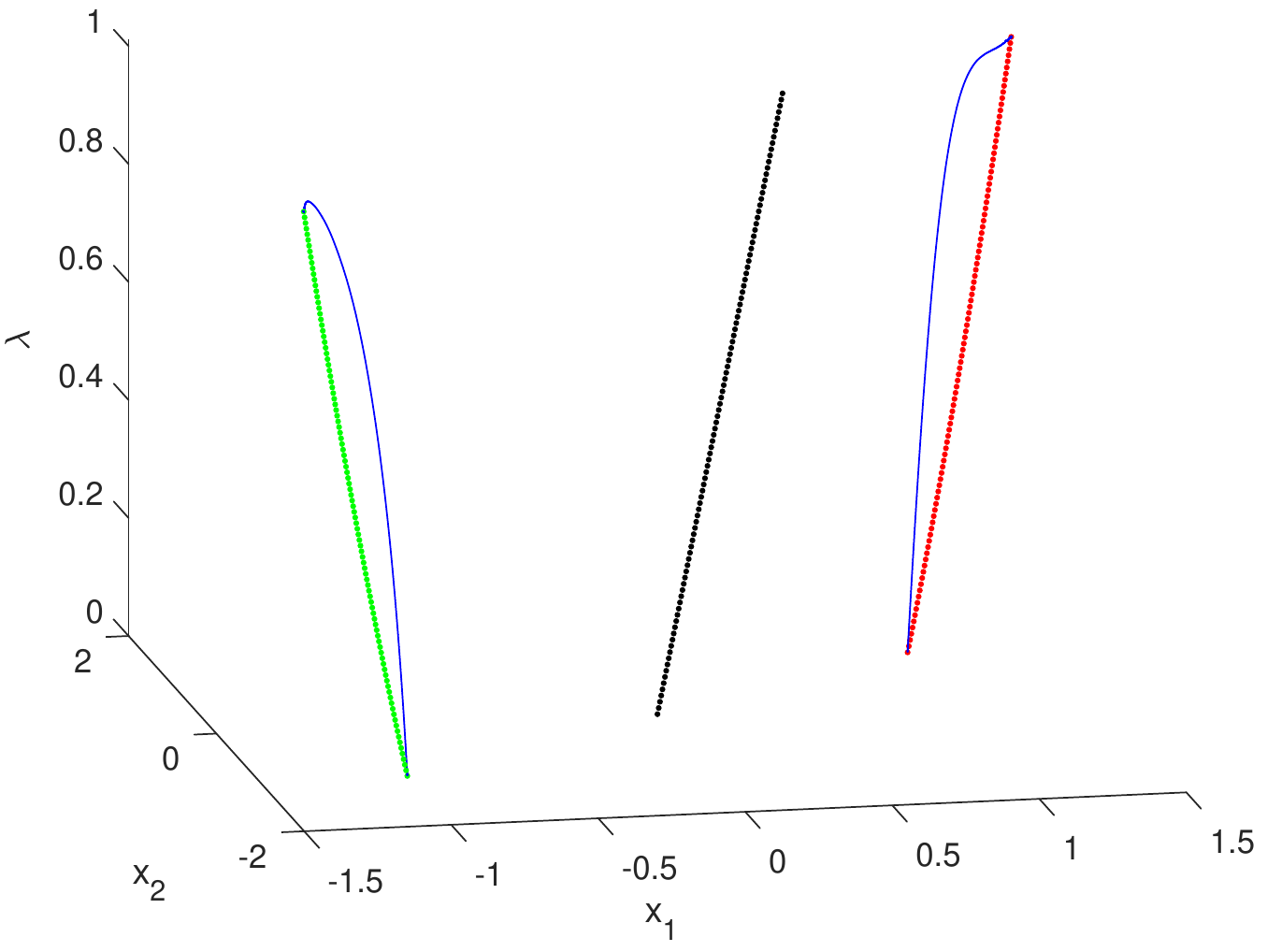}
		\caption{}
		\label{ex:nottips}
	\end{subfigure}
	\caption{In both pictures, the green, black, and red dots represent the positions of $p_1(s)$, $p_2(s)$, and $p_3(s)$, respectively, for different values of $\lambda = \Lambda(s)$. The blue curve is the pullback attractor to $p_{3-}$. In (a), the position of $\sqrt{b(\Lambda(s))+1}$ for small $s$-values is smaller than the position of $a(\Lambda(s))$ for larger $s$-values. Therefore, R-tipping is possible away from $p_{3-}$. In (b), $-\sqrt{b(\Lambda(s_1))+1} < a(\Lambda(s_2)) < \sqrt{b(\Lambda(s_1))+1}$ for all $s_1 < s_2$ (including in the infinite limits), so R-tipping is not possible away from either $p_{1-}$ or $p_{3-}$. Pullback attractors are shown in blue for both paths.}
	\label{fig:nDtipping}
\end{figure}

\section{Conclusion}\label{conclusion}

In summary, we have shown that R-tipping results are more complicated in multi-dimensional systems than in one-dimensional systems. R-tipping can happen anytime a path is not forward basin stable of a certain type, and sometimes there can be R-tipping even if a path is FBS. We proposed forward inflowing stability as a condition that prevents R-tipping in systems of all dimensions. One drawback is that it is difficult to know when a path is FIS because it requires knowledge about the autonomous systems with fixed parameter values and what sort of forward invariant sets exist around the equilibria. One future direction we could take is to give more concrete results about how to determine whether a path is FIS. \par

In this paper, we focused on stable paths of equilibria. However, in multi-dimensional systems, there could be stable paths of other attracting invariant sets, such as periodic orbits. In such a situation, different kinds of R-tipping are possible, known as partial tipping and full tipping (see \cite{hassan}). Not much work has been done to determine when these kinds of tipping can or cannot happen, but the FBS and FIS methods could be used as a starting point, as we believe the results given in this paper are generalizable to invariant sets other than fixed points.

\section*{Acknowledgements}

The authors acknowledge support for this research from the Office of Naval Research under grant N00014-18-1-2204.

\end{document}